\documentclass[11pt]{article}
\usepackage{amsmath}
\usepackage{amsfonts}
\usepackage{latexsym}
\usepackage{amssymb}
\usepackage{graphicx}
\usepackage[all]{xypic}
\usepackage{amsthm}
\usepackage{chemarrow}
\usepackage{verbatim}
\usepackage{color}
\usepackage{epsfig}
\usepackage{hyperref}
\usepackage{float}
\usepackage{enumerate}

\theoremstyle{plain}
\newtheorem{theorem}{Theorem}[section]
\newtheorem{corollary}[theorem]{Corollary}
\newtheorem{proposition}[theorem]{Proposition}
\newtheorem{lemma}[theorem]{Lemma}
\newtheorem{conjecture}[theorem]{Conjecture}

\theoremstyle{definition}
\newtheorem{definition}[theorem]{Definition}
\newtheorem{example}[theorem]{Example}
\newtheorem{notation}[theorem]{Notation}

\newtheorem{remark}[theorem]{Remark}
\newtheorem{claim}[theorem]{Claim}

\renewcommand\>{\rangle}
\newcommand\<{\langle}

\newcommand\ZZ{\mathbb{Z}}
\newcommand\Zplus{\mathbb{Z}_{>0}}
\newcommand\NN{\mathbb{Z}_{\geq 0}} 

\newcommand{\op}{\operatorname}
\renewcommand\iff{\Leftrightarrow}
\renewcommand\implies{\Rightarrow}

\newcommand{\bR}{\mathbb{R}}
\newcommand{\Rplus}{\bR_{>0}}
\newcommand{\Rnn}{\bR_{\geq 0}}

\newcommand\ol[1]{{\overline {#1}}}

\newcommand{\CRN}{chemical reaction network }

\newcommand{\CRNs}{chemical reaction networks }
\newcommand{\CRNsNoSpace}{chemical reaction networks}

\def\SS{\mathcal S}
\def\CC{\mathcal C}
\def\RR{\mathcal R}
\def\TT{\op{Super}}

\newcommand{\too}{\longrightarrow}

\newcommand{\lra}{\rightleftharpoons}

\newcommand{\source}[1]{\op{reactant}(#1)}
\newcommand{\target}[1]{\op{product}(#1)} 

\newcommand{\flux}[1]{\op{flux}(#1)} 

\newcommand{\init}[2]{\op{init}_{#1}(#2)}
\newcommand{\supp}[2]{\op{supp}_{#1}(#2)}
\newcommand{\GG}{(\SS,\CC,\RR)} 

\newcommand{\sphere}{S^{m-1}}
\newcommand{\cube}{[0,\infty]^m}
\newcommand{\bCoef}[2]{\beta_{#1}({#2})}

\newcommand{\relInt}{\op{int}}

\newcommand{\intsBdd}{\mathit{CmpctInt}}

\newcommand{\cpcOrth}{[0, \infty]^m}

\newcommand{\invtPoly}{\mathcal{P}}
\newcommand{\relIntP}{\relInt(\invtPoly)}

\newcommand{\HPerp}{H^{\perp}}
\newcommand{\Oset}{\mathcal{O}} 
\newcommand{\Kset}{\mathcal{K}} 
\newcommand{\pull}{\op{pull}}

\newcommand{\good}{sustaining }
\newcommand{\bad}{draining }

\newcommand{\Toric}{\op{Toric}}

\newcommand{\net}{ \GG }

\newcommand{\GAC}{GAC }
\newcommand{\GACnospace}{GAC}

\newcommand{\cutoff}{\theta}
\newcommand{\param}{\theta_0}
\newcommand{\wN}{w(i)}
\newcommand{\wI}{w(i)}
\newcommand{\thetaN}{\theta(i)}
\newcommand{\thetaI}{\theta(i)}
\newcommand{\toricJet}{\big(\theta(i)^{w(i)}\big)}
\newcommand{\unitJet}{(w(i))}

\newcommand{\sinpioverthree}{0.866025404}
\newcommand{\halfSinPiOverThree}{0.433012702}

   \def\fillrightmap{\mathord- \mkern-6mu
		     \cleaders\hbox{$\mkern-2mu \mathord-\mkern-2mu$}\hfill
		     \mkern-6mu \mathord\rightarrow}

\newenvironment{bullets}%
        {\begin{list}
                {\noindent\makebox[0mm][r]{$\bullet$}}
                {\leftmargin=5.5ex \usecounter{enumi}
 		 \topsep=1.5mm \itemsep=-.75ex}
        }
        {\end{list}}

\begin{document}

\title{A geometric approach to the Global Attractor Conjecture}
\author{Manoj Gopalkrishnan, Ezra Miller, and Anne Shiu}
\date{16 December 2013}

\maketitle

\begin{abstract}
This paper introduces the class of \emph{strongly endotactic
networks}, a subclass of the endotactic networks introduced by
G.~Craciun, F.~Nazarov, and C.~Pantea.  The main result states that
the global attractor conjecture holds for complex-balanced systems
that are strongly endotactic: every trajectory with positive initial
condition converges to the unique positive equilibrium allowed by
conservation laws.  This extends a recent result by D.~F.~Anderson for
systems where the reaction diagram has only one linkage class
(connected component).   The results here are proved using differential
inclusions, a setting that includes power-law systems.  The key
ideas include a perspective on reaction kinetics in terms of
combinatorial geometry of reaction diagrams, a projection argument
that enables analysis of a given system in terms of systems with lower
dimension, and an extension of Birch's theorem, a well-known result
about intersections of affine subspaces with manifolds parameterized
by monomials.
\end{abstract}

\tableofcontents


\section{Introduction}\label{s:intro}

The general study of reaction systems with mass-action kinetics goes
back at least to the pioneering work of Feinberg, Horn, and Jackson in
the 1970s~\cite{Feinberg72, HornJackson}.  The class of systems under
consideration in this paper includes such mass-action systems, as well
as power-law systems and S-systems from biochemical systems
theory~\cite{Savageau}.

Quite apart from the context of biochemistry, such systems appear to
be objects of intrinsic interest for engineers and mathematicians.
For instance, essentially identical models have been investigated in
genetic algorithms~\cite{QDS} and population
dynamics~\cite{Lotka1920}.  Models with similar mathematical structure
have been studied in distributed systems~\cite{Petri66} and algebraic
statistics~\cite{ASCB}.  Connections between these systems and
binomial algebra and toric geometry have been stressed by several
authors~\cite{Adleman2008, TDS, Craciun2010some, Gatermann, cat,
abelBinomial, ShiuSturmfels}.

The Global Attractor Conjecture (GAC) is the focus of this paper.
Given a reaction system, conservation laws induce a foliation of the
concentration space (which is a positive orthant) by polyhedra that
are forward-invariant with respect to the dynamics.  It is well-known
that when the dynamics of a reaction system admit a ``pseudo-Helmholtz
free energy function'' as a strict Lyapunov function, then each
forward-invariant polyhedron contains a special equilibrium point,
sometimes called the ``Birch point'' due to a connection to Birch's
theorem \cite{HornJackson}.  The GAC, which has resisted attempts at
proof for four decades~\cite{Horn74}, asserts that every trajectory
asymptotically approaches the unique Birch point in its
forward-invariant polyhedron.  A survey of literature relevant to the
\GAC appears in Section~\ref{s:conjs}.

The \GAC is usually stated for complex-balanced systems, which are
generalizations of the more well-known detailed-balanced systems.
Complex-balanced systems are known to admit pseudo-Helmholtz free
energy functions as Lyapunov functions.  Specializing the \GAC to the
case of networks in which every reaction has the form $A \to B$ (that
is, every complex is a species) yields the well-known Ergodic Theorem
for continuous-time, discrete-space, autonomous Markov
chains~\cite{Norris}.  From this perspective, the GAC may be viewed as
an ergodicity conjecture for chemical reaction networks evolving under
mass-action kinetics.  Indeed, reaction systems theory may be viewed
as a nonlinear generalization of the theory of continuous-time
discrete-space autonomous Markov chains.

A key idea in our work is that the combinatorial geometry of a
reaction network in the space of (chemical) complexes---the reactants
and products of the reactions---informs the dynamics in concentration
space.  This connection was anticipated to a certain extent by the
``extended permanence conjecture''
(Conjecture~\ref{conj:ext_permanence}) of Craciun, Nazarov, and
Pantea, which implies the \GACnospace~\cite{CNP}.  Their conjecture
captures the intuitively appealing idea that if a reaction network
``points inwards'', in the sense of being endotactic (see
Definition~\ref{d:endotactic}), then the corresponding dynamics in
concentration space must also roughly ``point inwards'', in the sense
of being permanent (see Definition~\ref{d:persistent_permanent}).

We develop the correspondence between geometry of a network in the
space of complexes and dynamics in concentration space by analyzing
the contributions of each reaction to the dynamics along ``toric
jets'' (Definition~\ref{d:jet}.\ref{d:toricjet}).  In doing so, we
positively resolve the extended permanence conjecture for a subclass
of endotactic networks called ``strongly endotactic'' networks
(Definition~\ref{d:endotactic}), our first main result.


\begin{theorem}\label{thm:main}
Every strongly endotactic reaction network is permanent.
\end{theorem}

As stated above, the extended permanence conjecture implies the
\GACnospace.


\begin{theorem}\label{thm:main2}
The \GAC holds for strongly endotactic complex-balanced reaction
systems.
\end{theorem}

The fact that weakly reversible networks with one linkage class
(i.e., with strongly connected reaction graph) are strongly
endotactic (Corollary~\ref{cor:one_l_class}) yields an easy
consequence.


\begin{theorem}\label{thm:OneLClass}
Every weakly reversible reaction network with exactly one linkage
class is permanent.
\end{theorem}

Theorem~\ref{thm:OneLClass} strengthens a result in which Anderson
proved persistence of weakly reversible systems with exactly one
linkage class under the assumption that all trajectories are
bounded~\cite{Anderson11}.  Our approach was significantly influenced
by distillation of ideas from Anderson's result, combined with
insights gained from~\cite{CNP}.

The class of dynamical systems considered here includes those known as
generalized mass-action systems, or power-law systems, studied in
biochemical systems theory~\cite{Savageau}.  More precisely, we prove
our results here in the more general setting of mass-action
differential inclusions, which we introduced in earlier
work~\cite{ProjArg}.  The crucial result from \cite{ProjArg} applied
here (in Section~\ref{s:results}) concerns families of differential
inclusions that are closed under certain projections, allowing us to
analyze a system in terms of systems of lower dimension.

In the course of this work (Section~\ref{s:extBirch}), we prove two
extensions of Birch's theorem, a well-known result in reaction network
theory~\cite{HornJackson} and algebraic statistics~\cite{Birch}
concerning intersections of certain affine spaces and manifolds
parametrized by monomials (Theorems~\ref{t:birch}
and~\ref{t:birchRestated}).  We show that Birch's theorem remains true
under slight perturbation (Theorem~\ref{thm:extBirch}), and at
infinity in suitable compactifications (Theorem~\ref{t:birchExt1}).

To connect the combinatorial (polyhedral) geometry of a reaction
diagram to the asymptotics of its dynamics (Section~\ref{s:jet}),
we develop jets and toric jets (Definition~\ref{d:jet}), in parallel with
the notion of ``jet frame'' from work by Miller and Pak \cite{MP} on
unfolding convex polyhedra.  Jets capture motion the bulk of which
proceeds in a fixed main direction, but is perturbed to first order in
an orthogonal direction, to second order in a third direction, and so
on.  Jets coherently tease apart the contributions of various
reactions to the gradient of pseudo-Helmholtz free energy along
infinite trajectories~(Proposition~\ref{p:domination}).

Birch's Theorem and jets come together (in Section~\ref{s:prelya}) to
show that for strongly endotactic networks, within each stoichiometric
compatibility class there exists a compact set outside of which the
pseudo-Helmholtz free energy function $\sum_{i\in\SS} (x_i \log
x_i - x_i)$ decreases along trajectories
(Theorem~\ref{thm:LyFncWorks_cpc}).  Our main results on persistence
and permanence, described earlier in this Introduction, follow in
Section~\ref{s:results}.

Examples in Section~\ref{s:ex} illustrate our results and explain the
limitations of our approach.  A further extended example in
Section~\ref{s:ped_intro} serves to introduce the main ideas, after
which we give precise, general definitions concerning reaction
networks and their accompanying mass-action differential inclusions
(Section~\ref{s:CRNT}).  Various conjectures related to the
\GACnospace, along with known partial results, are collected in
Section~\ref{s:conjs}.


\section{An illustrative example}\label{s:ped_intro}


Consider the following reaction network with two species $X$ and $Y$
and three reactions:
\begin{align}\label{ntwk:RLV}
  2X\stackrel{k_1}\too X
  \qquad
  0\stackrel{k_2}\too Y
  \qquad
  2Y\stackrel{k_3}\too X+Y,
\end{align}
where $k_1,k_2, k_3 \in \Rplus$ denote the reaction rate constants.
Network~\eqref{ntwk:RLV} is obtained by reversing all the reactions in
the well-known Lotka--Volterra reaction network.  Letting $x(t)$ and
$y(t)$ denote concentrations of $X$ and $Y$, respectively, at
time~$t$, network~\eqref{ntwk:RLV} defines the following system of
ordinary differential equations arising from mass-action kinetics:
\begin{equation}\label{deq1}
\begin{pmatrix} \dot{x} \\ \dot{y} \end{pmatrix}
  = k_1x^2\begin{pmatrix} -1 \\ 0 \end{pmatrix}
    + k_2\begin{pmatrix} 0 \\ 1 \end{pmatrix}
    + k_3y^2\begin{pmatrix} 1 \\ -1 \end{pmatrix}.
\end{equation}

In this section, we introduce the main ideas of our work by explaining
how to prove that network~\eqref{ntwk:RLV} taken with mass-action
kinetics is \emph{permanent}: there exists a compact set $K \subseteq
\Rplus^2$ such that every trajectory of the dynamical
system~\eqref{deq1} in $\Rplus^2$ eventually remains in $K$.  We begin
with the following assertion.

\begin{claim}\label{claim:ex}
There exists a compact set $K \subseteq \Rplus^2$ (that depends on
$k_1, k_2, k_3$) such that outside $K$, the function
$$%
  g(x,y) = x \log x - x + y \log y - y
$$
on $\Rplus^2$ is strictly decreasing along trajectories
$\big(x(t),y(t)\big)$ of~\eqref{deq1} except at an equilibrium.
\end{claim}

By itself, Claim~\ref{claim:ex} does not imply that the dynamical
system~\eqref{deq1} is permanent.  For instance, consider the figure
$$%
  \includegraphics[width=1.2in]{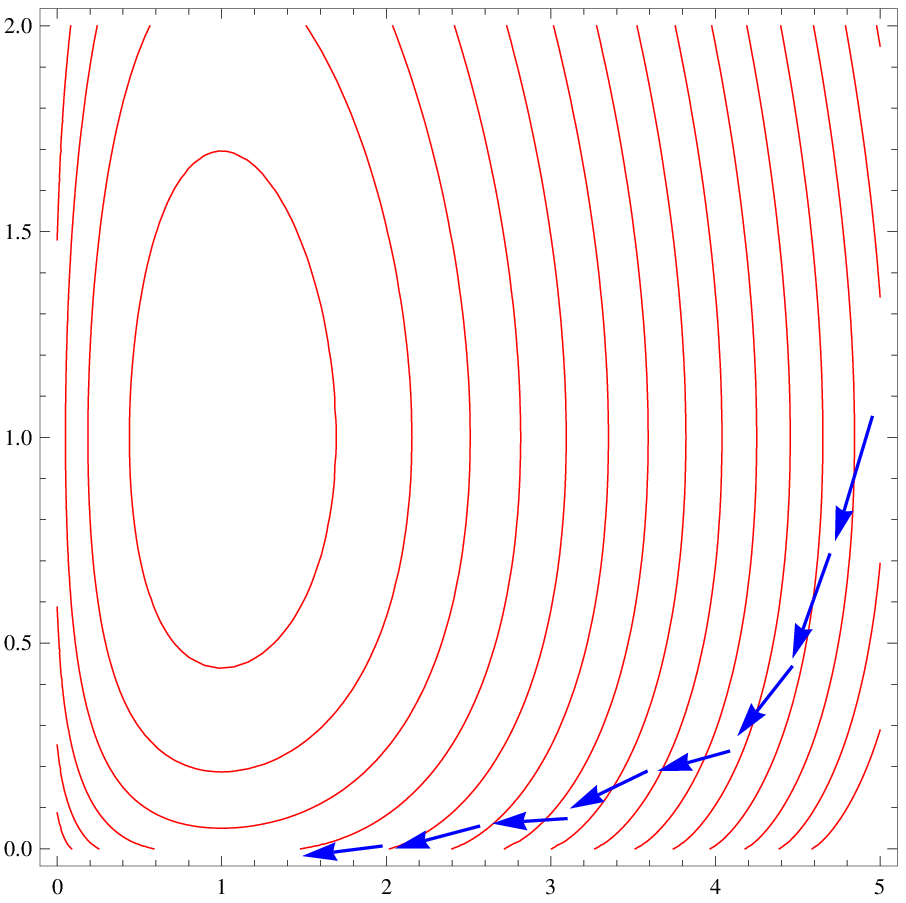}
$$
in which level sets of $g$ are indicated in red, with interior level
sets correspond to lower values of~$g$.  The function $g$ decreases
along the trajectory indicated in blue, but this trajectory is not
even persistent---that is, at least one coordinate (in this case, the
$y$-coordinate) approaches zero.

Nevertheless, Claim~\ref{claim:ex} is enough to establish that the
origin $(0,0)$ is repelling
(Definition~\ref{d:persistent_permanent}.2) and that all trajectories
of~\eqref{deq1} are bounded.  These two properties, along with good
behavior under projection, suffice to obtain our desired permanence
result.  This follows from our earlier work~\cite{ProjArg}, as
explained in Sections~\ref{sub:pers}--\ref{sub:perm}, particularly
Lemma~\ref{l:projArg}.  As a consequence of permanence, the blue
trajectory depicted here cannot be a trajectory~of~\eqref{deq1}.

Our explanation of Claim~\ref{claim:ex} involves a ``proof by
picture''.  \emph{Concentration space} refers to the space in which
the trajectories of the dynamical system~\eqref{deq1} evolve,
excluding points where $X$ or~$Y$ has concentration zero; thus
concentration space is $\Rplus^2$.  \emph{Energy space} has
coordinates $u = \log x$ and $v = \log y$, so energy space is
(another) $\bR^2$.  Concentration space and energy space are
diffeomorphic via the Lie group isomorphism $(x,y) \mapsto (\log x,
\log y)$.  This map sends the identity $(1,1)$ to the origin $(0,0)$,
and the curve parametrized by $\theta \mapsto (\theta^a,\theta^b)$ in
concentration space to the ray from the origin in direction $(a,b)$ in
energy space:
$$%
\begin{array}{ccc}
  \begin{array}{c}
  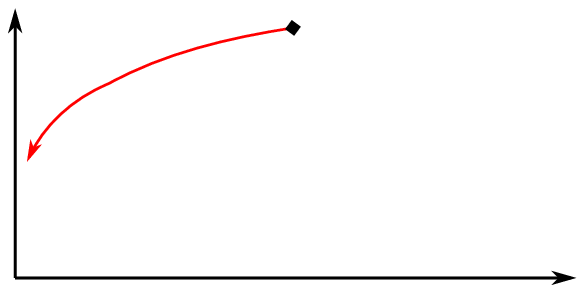\\
  \end{array}
&
  \begin{array}{c}
  \stackrel{(\log,\log)}{\fillrightmap}
  \end{array}
&
  \begin{array}{c}
  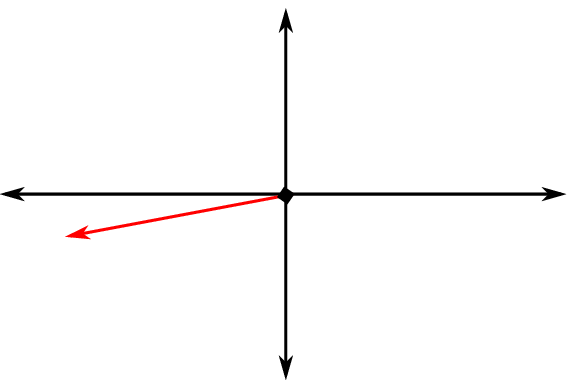\\
  \end{array}
\\
  \text{toric ray in concentration space}
&
&
  \text{a vector in energy space}
\end{array}
$$
Such curves in concentration space, obtained by exponentiating rays in
energy space, are central to our analysis.  The curve parametrized by
$(\theta^a, \theta^b)$ is the \emph{toric ray} in direction~$(a,b)$.

\begin{remark}\label{r:exponentiation}
In Lie theory or differential geometry, the inverse map from energy
space to concentration space is the usual exponential map from the
tangent space at the identity $(1,1)$ to the group.  Thus toric rays
are the positive parts of 1-parameter subgroups of concentration
space, and they can also be thought of as geodesics.
\end{remark}

For every point $(x,y) \in \Rplus^2$, the gradient of $g$ at $(x,y)$
is parallel to the direction of the toric ray that passes through that
point.  This is because $\nabla g(x,y) = (\log x, \log y)$, so that
\begin{align}\label{eqngtoric}
\nabla g(\theta^a,\theta^b)=(\log \theta)(a,b) .
\end{align}

Combinatorial information about a reaction network is represented in
the \emph{space of complexes} by the reaction diagram, which is a
geometric representation of the reaction network.  The reaction
diagram for network~\eqref{ntwk:RLV} is depicted here:
$$%
  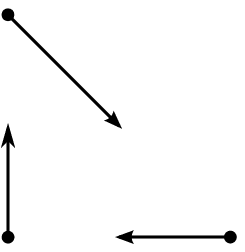
$$
The reaction $pX + qY \stackrel k\too rX + sY$, for example, is
represented by the arrow from the complex $(p,q)$ to the complex
$(r,s)$.  (Other authors have sometimes used the term ``space of
complexes'' to refer to a real vector space of dimension equal to the
number of vertices of the reaction graph.  Readers should be aware
that our usage is different.  For us, the dimension of the space of
complexes equals the number of species.)

The right-hand side of the differential equations~\eqref{deq1}
consists of a sum in which the summand arising from a reaction of the
form $pX + qY \stackrel k\too rX + sY$ is:
\begin{align}\label{contr}
kx^py^q \begin{pmatrix} r-p \\ s-q \end{pmatrix}.
\end{align}
The relation between concentration space and the space of complexes
depicted in Figure~\ref{figrlv} is central to our analysis.
Specifically, along a toric ray parametrized by $(\theta^a,
\theta^b)$, the monomial contribution $x^py^q$ of the reaction $pX +
qY \stackrel k\too rX + sY$ equals $\theta^{\< (a,b), (p,q)\>}$.  The
exponent $\<(a,b), (p,q)\>$ has a geometric interpretation in the
space of complexes as the value of the linear functional $ax+by$ at
the point $(p,q)$.  We now ask: as $\theta\to+\infty$, which reactions
are dominant, that is, which reactions yield the largest-magnitude
contribution~\eqref{contr}?  The answer is that the dominant reactions
are the ones with maximal inner product $\< (a,b), (p,q)\>$, because
their monomial contributions grow the fastest as $\theta \to + \infty$
and thus overwhelm the lesser monomial contributions.  The directions
in which these reactions ``pull'' (that is, the vector contributions
to the differential equations~\eqref{deq1}) remain unchanged
along the toric ray and are given by their \emph{reaction vectors}
$\begin{pmatrix} r-p \\ s-q
\end{pmatrix}$ as in~\eqref{contr}.

\begin{figure}
$$%
  \begin{array}{c}
  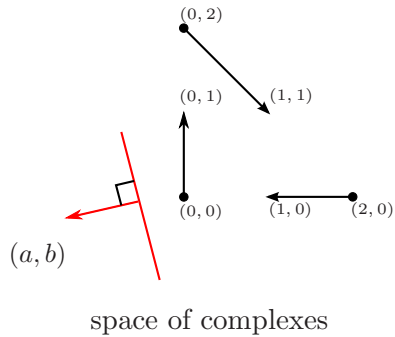\\[1ex]
  \text{space of complexes}\\[-3ex]
  \end{array}
$$
\caption{\label{figrlv}``Proof by picture'' for permanence of the
reverse Lotka-Volterra system.  The toric ray corresponding to the
depicted vector pulls hardest on the source of the vertical~reaction.}
\end{figure}

We emphasize this key point: along every toric ray, there exist some
asymptotically dominant reactions that determine the dynamics.  These
reactions can be determined purely combinatorially: they are the
reactions $pX + qY \to rX + sY$ whose \emph{reactant complexes}
$(p,q)$ attain the maximal inner product with the direction $(a,b)$ of
the toric ray.

Returning to Claim~\ref{claim:ex}, the
time derivative $\frac{d}{dt} g(x(t),y(t))$ of $g$ at a trajectory
point at time $t = t_0$ of the form $(x(t_0),y(t_0)) =
(\theta^a,\theta^b)$ is determined by the chain rule~to~be
\begin{align}\label{eq:g-deriv}
\frac{d}{dt} g(x(t),y(t))\Big|_{t=t_0}
  = \big\<\nabla g(\theta^a, \theta^b), P(\theta^a, \theta^b)\big\>
\end{align}
where $P(x,y)$ denotes the right-hand side of the differential
equations~\eqref{deq1} evaluated at the point $(x,y)$.  By linearity,
we can analyze separately the contribution to the
derivative~\eqref{eq:g-deriv} of each reaction $pX + qY \stackrel k\too
rX + sY$.  The contribution of such a reaction, from
equations~\eqref{eqngtoric} and~\eqref{eq:g-deriv}, is
\begin{align}\label{eq:contr}
\<(a,b),(r-p,s-q)\>
\theta^{\<(a,b),(p,q)\>} k\log \theta.
\end{align}

So far, everything stated above holds for any reaction network with
two species, and can be appropriately generalized for more species.
Now we appeal to the fact that the reactions in the reaction diagram
under consideration ``point inward'' (i.e., the network is
\emph{strongly endotactic}; see
Definition~\ref{d:endotactic}.\ref{d:strong_endotactic}).

The term $\log\theta$ in the contribution~\eqref{eq:contr} is common
to every reaction, so it is not significant to the analysis of the
sign of the derivative $\frac{d}{dt} g(x(t),y(t))$.  On the other hand,
the reaction rate constants $k_1, k_2, k_3 \in \Rplus$ are significant.
However, it turns out that we can ignore the rate constants, and
instead simply analyze
$$%
  \pull = \<(a,b),(r-p,s-q)\>\theta^{\<(a,b),(p,q)\>},
$$
which we call the \emph{pull} (see Definition~\ref{d:pull}) of the
reaction along the toric ray in direction $(a,b)$.  Details appear in
Sections~\ref{s:prelya} and~\ref{s:results}, particularly in the
proof of Theorem~\ref{thm:LyFncWorks_cpc}.

Suppose we could verify by inspection of the reaction diagram that for
every direction $(a,b)$, the inner products in the pulls of all
dominant reactions along toric rays in direction $(a,b)$ satisfy the
inequality $\<(a,b),(r-p,s-q)\> < 0$.
By compactness of the unit circle, whose points are thought of as
directions, we could choose a uniform cutoff $\theta$ large enough so
that along every toric ray, past this cutoff the monomial
contributions of dominant reactions overwhelm the contributions of all
other reactions.  Indeed, we could ensure this even after including the
effects of reaction rate constants.  Hence, outside the compact set
$$%
  K_\theta= \big\{ (\theta_0^a,\theta_0^b)\mid (a,b)\text{ is a unit
  vector in energy space and }\theta_0\in[1,\theta] \big\}
  \subseteq \Rplus^2
$$
in concentration space, the
time derivative of~$g$ would be negative.

Indeed, in generic directions $(a,b)$, the dominant reactions along
the toric ray in direction $(a,b)$ do satisfy the required inequality.
However, this fails in precisely three directions, namely, $(-1,0)$,
$(0,-1)$, and $(1,1)$.

For instance, consider the direction $(a,b) = (0,-1)$.  There exist
two dominant reactions: one ``sustaining'' reaction $0
\stackrel{k_2}\too Y$ with $\<(a,b),(r-p,s- q)\> < 0$, and another
reaction $2X \stackrel{k_1}\too X$ with $\<(a,b),(r-p,s- q)\> = 0$.
The reaction $2X \stackrel{k_1}\too X$ makes no contribution to the
derivative of $g$ in direction $(0,-1)$.  However, consider a nearby
direction $(-\varepsilon, -1)$ for some small $\varepsilon > 0$.  In
this direction, the reaction $2X \stackrel{k_1}\too X$ is now
``draining'': its pull is strictly greater than $0$.  It is true that
if $\varepsilon > 0$ is fixed, then along the toric ray in direction
$(-\varepsilon, -1)$, this reaction is eventually dominated by the
sustaining reaction, because its monomial term is now smaller.
However, as $\varepsilon$ gets smaller, the value of the cutoff
$\theta$ after which this domination of monomials occurs must become
arbitrarily large.

This is problematic, because the compact set $K_\theta$ requires a
single value of~$\theta$ to work as a cutoff for every direction
$(a,b)$.  To accomplish this, we turn to a second observation: for
sufficiently small $\varepsilon > 0$, the inner product of the
reaction vector of \mbox{$2X \stackrel{k_1}\too X$} with the direction
$(-\varepsilon, -1)$ of the toric ray is, although positive, near
zero.  A more detailed analysis along these lines, using information
from both the monomial and the inner product, accomplishes
Claim~\ref{claim:ex}.  This approach is developed in
Section~\ref{s:jet} via the technology of jets and jet frames.
These allow us to reduce the analysis of the sign of the
time derivative of~$g$ to a combinatorial calculation on the reaction
diagram.

Returning to the example, the family of directions $(-\varepsilon,-1)$
as $\varepsilon \to 0^+$ corresponds to the jet frame
$\big((0,-1),(-1,0)\big)$.  Two reactions dominate along the toric ray
in direction $(0,-1)$, namely $2X \to X$ and $0 \to Y$.  As for the
second direction $(-1,0)$ of the jet frame,
only the sustaining reaction $0 \to Y$ is dominant in that direction.
The pull of the draining reaction $2X \to X$ is dominated by the pull
of the sustaining reaction $0 \to Y$ in a uniform manner in all
directions $(-\varepsilon,-1)$ for small $\varepsilon > 0$, and that
produces the required uniform~cutoff~$\theta$.

In general,
for strongly endotactic reaction networks,
along every jet frame the pull of each draining reaction is dominated
by the pull of some sustaining reaction~(Proposition~\ref{p:domination}).
This result is key to obtaining a general version of
Claim~\ref{claim:ex}, which we then apply to prove our main results.


\section{Reaction network theory}\label{s:CRNT}

In this section, we recall the definitions of reaction networks and
their associated mass-action differential inclusions, following the
notation in our earlier work~\cite{ProjArg}.

\subsection{Reaction networks and reaction systems}\label{sub:networks}


\begin{definition}\label{d:crn}
A \emph{reaction network} $(\SS, \CC, \RR)$ is a triple of finite
sets: a set $\SS$ of \emph{species}, a set $\CC \subseteq \bR^\SS$ of
\emph{complexes}, and a set $\RR \subseteq \CC \times \CC$ of
\emph{reactions}.  The \emph{reaction graph} is the directed graph
$(\CC,\RR)$ whose vertices are the complexes and whose directed edges
are the reactions.  A reaction $r = (y,y') \in \RR$, also written $y
\to y'$, has \emph{reactant} $y = \source{r} \in \bR^\SS$,
\emph{product} $y' = \target{r} \in \bR^\SS$, and \emph{reaction
vector}
$$%
  \flux{r} = {\target{r} - \source{r}} = y' - y.
$$
A \emph{linkage class} is a connected component of the reaction graph.
The \emph{reaction diagram} is the realization $(\CC,\RR) \to \bR^\SS$
of the reaction graph that takes each reaction $r \in \RR$ to the edge
from $\source r$ to $\target r$.  The \emph{reactant polytope} is the
convex hull ${\op{Conv}\{y\in\bR^\SS\mid y \to y'\in\RR\}}$ of the
reactant complexes.
\end{definition}

Beginning in the following example, we follow the usual conventions of
depicting a network by its reaction graph or reaction diagram and
writing a complex as, for example, $2A+B$ rather than $y=(2,1)$.


\begin{example}\label{ex:end_intro}
The following network has two species ($A$ and $B$), five complexes,
four reactions (each indicated by a unidirectional arrow), and two
linkage classes:
$$%
  2A \lra A+B \quad \quad B \to 0 \to 2B .
$$
The reaction polytope is the convex hull of the four reactant complexes $(2,0),(1,1),(0,1)$, and $(0,0)$. 
\end{example}


\begin{remark}\label{rmk:genl_crn}
The \CRN theory literature usually imposes the following requirements
for a reaction network.
\begin{itemize}
\item%
Each complex takes part in some reaction: for all $y \in \CC$ there
exists $y' \in \CC$ such that $(y,y') \in \RR$ or $(y',y) \in \RR$;
and
\item%
no reaction is trivial: $(y,y)\notin \RR$ for all $y \in \CC$.
\end{itemize}
Definition~\ref{d:crn} does \emph{not} impose these conditions: our
reaction graphs may include isolated vertices or self-loops.  In our
earlier work, we dropped these conditions to ensure that the
projection of a network---obtained by removing certain
species---remains a network under our definition even if some
reactions become trivial~\cite{ProjArg}.  In addition, like Craciun,
Nazarov, and Pantea~\cite[\S 7]{CNP}, we allow arbitrary real
complexes $y \in \bR^\SS$.  Thus our setting is more general than that
of usual \CRNsNoSpace, whose complexes $y \in \NN^\SS$ are nonnegative
integer combinations of species, as in the next definition.  The ODE
systems defined in \S\ref{s:m-a} that result from real
complexes have been studied over the years and called ``power-law
systems''.
\end{remark}


\begin{definition}\label{d:network}
A reaction network $(\SS,\CC,\RR)$ is
\begin{enumerate}
\item%
\emph{chemical} if $\CC \subseteq \NN^\SS$;
\item%
\emph{reversible} if the reaction graph of the network is undirected:
a reaction $(y,y')$ lies in~$\RR$ if and only if its reverse reaction
$(y',y)$ also lies in~$\RR$;
\item%
\emph{weakly reversible} if every linkage class of the network is
strongly connected.
\end{enumerate}
\end{definition}

\begin{definition}\label{d:stoichiometry}
The \emph{stoichiometric subspace} $H$ of a network is the span of its
reaction vectors.  The \emph{dimension} of a network is the dimension
of its stoichiometric subspace $H$.  For a positive vector $x_0 \in
\Rplus^\SS$, the \emph{invariant polyhedron} of $x_0$ is the
polyhedron
$$%
  \invtPoly = (x_0 + H) \cap \Rnn^\SS.
$$
This polyhedron is also referred to as the \emph{stoichiometric
compatibility class} in the \CRN theory literature~\cite{Fein87}.
\end{definition}

\begin{example}\label{ex:end}
Recall the network from Example~\ref{ex:end_intro} with reactions $2A
\lra A+B$ and $B \to 0 \to 2B$.  This is a two-dimensional chemical
reaction network that is not weakly reversible.  For every choice of $x_0\in\Rplus^2$, the corresponding invariant
polyhedron is the positive orthant: $\invtPoly = \Rnn^2$.
\end{example}

Another polyhedron of interest appears in the next definition.  For an
introduction to polyhedral geometry, we refer the reader to the book
by Ziegler~\cite{Ziegler}.

\begin{definition}\label{d:reactant polytope}
For a positive integer $n\in\Zplus$, a \emph{polytope} in $\bR^n$ is
the convex hull of a finite set of points in $\bR^n$.  The
\emph{reactant polytope} of reaction network $\GG$ is the convex hull
of the reactant complexes $\source{\RR} \subseteq \bR^\SS$.
\end{definition}


We now turn to the concept of a reaction system.

\begin{definition}\label{d:reaction_sys}
Write $\intsBdd = \big\{[a,b] \mid 0 < a \leq b < \infty\big\}$ for
the set of compact subintervals of~$\Rplus$.  Let $\GG$ be a reaction
network.  A \emph{tempering} is a map $\kappa: \RR \to \intsBdd$ that
assigns to each reaction a nonempty compact positive interval.  A
\emph{confined reaction system} consists of a reaction network $\GG$, a
tempering $\kappa$, and an invariant polyhedron $\invtPoly$ of the
network.
\end{definition}

\subsection{Strongly endotactic \CRNs}\label{str.endotactic}

Endotactic \CRNsNoSpace, a generalization of weakly reversible
networks, were introduced by Craciun, Nazarov, and Pantea~\cite[\S
4]{CNP}.  We introduced strongly endotactic networks, a subclass of
endotactic networks, in~\cite{ProjArg}.  We now recall the
definitions.


\begin{definition}\label{d:partialOrder}
The standard basis of $\bR^\SS$ indexed by~$\SS$ defines a canonical
inner product $\< \cdot,\cdot \>$ on $\bR^\SS$ with respect to which
the standard basis is orthonormal.  Let $w \in \bR^\SS$.
\begin{enumerate}
\item%
The vector $w$ defines a preorder on $\bR^\SS$, denoted by $\leq_w$, in
which
$$%
  y \leq_w y' \ \iff\ \< w,y \> \leq \< w,y' \>.
$$
Write $y <_w y'$ if $\< w,y \> < \< w,y' \>$.
\item%
For a finite subset $Y \subseteq \bR^\SS$, denote by $\init wY$ the set
of $\leq_w$-maximal elements of~$Y$:
$$%
  \init wY = \big\{y \in Y \mid \< w,y \> \geq \< w,y'
  \> \text{ for all } y' \in Y \big\}.
$$
\item%
For a reaction network $(\SS,\CC,\RR)$, the set $\RR_w \subseteq \RR$ of
$w$-\emph{essential reactions} consists of those whose reaction
vectors are not orthogonal to~$w$:
$$%
  \RR_w = \big\{r \in \RR \mid \< w,\flux{r} \> \neq 0\big\}.
$$
\item%
The $w$-\emph{support} $\supp w{\SS,\CC,\RR}$ of the network is the
set of vectors that are $\leq_w$-maximal among reactants of
$w$-essential reactions:
$$%
  \supp w{\SS,\CC,\RR} = \init w{\source{\RR_w}}.
$$
\end{enumerate}
\end{definition}

\begin{remark}\label{rmk:switchSignEndo}
In order to simplify the computations in Section~\ref{s:jet}, we
differ from the usual convention~\cite{CNP,Pantea}, by letting
$\init wY$ denote the $\leq_w$-maximal elements rather than the
$\leq_w$-minimal elements.  Accordingly, the inequalities in
Definition~\ref{d:endotactic} are switched, so our definition of
endotactic is equivalent to the usual one.
\end{remark}


Before presenting Definition~\ref{d:endotactic}, 
we provide some underlying geometric intuition,
first in terms of 1-dimensional
projections (Remarks~\ref{rmk:endo1D} and~\ref{rmk:project}) and then
via reactant polytopes (Remark~\ref{rmk:endo_polytope}).  A third
interpretation via jet frames appears later in our work
(Lemma~\ref{l:endjets} and Proposition~\ref{p:stendjets}).


\begin{remark}\label{rmk:endo1D}
For a 1-dimensional network,
whose reaction diagram $(\CC,\RR)$ lies on a line in $\bR^\SS$, let $w
\in \bR^\SS$ be a nonzero vector that generates the stoichiometric
subspace.  Every nontrivial reaction $y \to y'$ either \emph{points to
the right} (points along $w$) or \emph{points to the left} (points
along $-w$).

\begin{enumerate}
\item%
$\GG$ is endotactic if and only if each nontrivial reaction with a
leftmost ($\leq_w$-minimal) reactant points to the right, and each
nontrivial reaction with a rightmost ($\leq_w$-maximal) reactant
points to the left~(see the bottom of Figure~\ref{fig:end}).

\item%
A 1-dimensional endotactic network that lies on a line is strongly
endotactic if and only if there exists a nontrivial reaction $y \to
y'$ (which necessarily points to the right) whose reactant $y$ is a
leftmost reactant and additionally there exists a nontrivial reaction
$z \to z'$ whose reactant is rightmost.  See the bottom of
Figure~\ref{fig:end}.
\end{enumerate}
\end{remark}


\begin{remark}\label{rmk:project}
We learned from Craciun and Pantea the following 
intuition behind 
Definition~\ref{d:endotactic} in terms of 1-dimensional projections.
Consider a reaction network $\GG$ and a line generated by a nonzero
vector $w\in\bR^\SS$.  The orthogonal projection of the reaction
diagram $(\CC,\RR)$ onto the line is the reaction diagram of a
dimension 0 or 1 network whose complexes are the projections $\<w,y\>
w$ for $y \in \CC$.  The network $\GG$ is endotactic (respectively,
strongly endotactic) if and only if for all vectors $w\in\bR^\SS$ that
are not orthogonal to the stoichiometric subspace, the projection of
the network onto the line generated by $w$ is endotactic
(respectively, strongly endotactic).  See Figure~\ref{fig:end}.  The
dual picture to these ideas was explained
in~\cite[Proposition~4.1]{CNP} by way of the so-called
``parallel~sweep~test.''
\end{remark}


\begin{figure}[ht]
\mbox{}\\[-15ex]\mbox{}
\begin{align*}
  \begin{xy}<14mm,0cm>:
  (-1.75,0) 	="left" *+!R{ }  *{}; 
  (6,0) 	="right" *+!R{ }  *{}; 
  (0,-.25) 	="down" *+!R{ }  *{}; 
  (0,3.75) 	="up" *+!R{ }  *{}; 
  (1,1) 	="y1" *{\bullet} *+!U{y_1}  *{}; 
  (1,2) 	="y1'" *+!D{y_1'  }  *{}; 
  (2.25,1) 	="y2"*{\bullet} *+!UR{y_2 }  *{}; 
  (3,2) 	="y2'" *+!DL{y_2'}  *{}; 
  (5.5,1.5) 	="y4" *{\bullet}*+!L{ y_3}  *{}; 
  (4.75,0.75) 	="y4'" *+!R{y_3' }  *{}; 
  (-1.75,1.6) 	="w1" *+!UL{}  *{}; 
  (-.75,1.6) 	="w2" *+!R{ }  *{}; 
  (-0.25,0.45) 	="X1" *+!R{ }  *{}; 
  (6.4,0.45) 	="X2" *+!R{ }  *{}; 
  (6.4,2.65) 	="X3" *+!UR{ G }  *{}; 
  (-.25,2.65) 	="X4" *+!R{ }  *{}; 
  (-.25,-.9) 	="Y1" *+!R{ }  *{}; 
  (6.4,-.9) 	="Y2" *+!R{ }  *{}; 
  (6.4, -.15) 	="Y3" *+!UR{ N }  *{}; 
  (-.25,-.15) 	="Y4" *+!R{ }  *{}; 
  (1,-.5) 	="y1proj" *{\bullet}*+!R{\widetilde{y_1} =\widetilde{y_1}' ~ }*{}; 
  (1,-.57) 	="y1proj_2" *+!U{\circlearrowleft }  *{}; 
  (2.25,-.5) 	="y2proj" *{\bullet}*+!R{\widetilde{y_2} }  *{}; 
  (3,-.5) 	="y2'proj" *+!L{\widetilde{y_2}'}  *{}; 
  (5.5,-.5) 	="y4proj" *{\bullet}*+!L{\widetilde{y_3}}  *{}; 
  (4.75,-.5) 	="y4'proj" *+!R{\widetilde{y_3}'}  *{}; 
  {\ar "y1";"y1'"*{}  }; 		
  {\ar "y2";"y2'"*{}  }; 		
  {\ar "y4";"y4'"*{}  }; 	
  {\ar "y2proj";"y2'proj"*{}  }; 		
  {\ar "y4proj";"y4'proj"*{}  };
  {\ar@{-->}^{w}"w1";"w2"  *{}  }; 		
  "X1";"X2" **\dir{.};
  "X2";"X3" **\dir{.};
  "X3";"X4" **\dir{.};
  "X4";"X1" **\dir{.};
  "Y1";"Y2" **\dir{.};
  "Y2";"Y3" **\dir{.};
  "Y3";"Y4" **\dir{.};
  "Y4";"Y1" **\dir{.};
  \end{xy}
\mbox{}\\[-8ex]\mbox{}
\end{align*}
\caption{At the top, we depict a direction vector $w = (1,0)$ and the
reaction diagram of a reaction network $G$ with three reactions $y_i
\to y_i'$.  At the bottom is the reaction diagram of the projection of
$G$ to the line generated by $w$.  This 1-dimensional network $N$ is
endotactic as explained in Remark~\ref{rmk:endo1D}: the leftmost
reactant of the nontrivial reactions is $\widetilde{y_2}$ and
$\widetilde{y_2} \to \widetilde{y_2}'$ points to the right, and the
rightmost reactant complex of the nontrivial reactions is
$\widetilde{y_3}$ and $\widetilde{y_3} \to \widetilde{y_3}'$ points to
the left.  So, $G$ is $w$-endotactic (but not endotactic: consider the
vector $w'=(-1,1)$).  $G$ is not strongly endotactic because
$\widetilde{y_1}$ is the unique leftmost reactant of $N$ but
$\widetilde{y_1} \to \widetilde{y_1}'$ does not point to the right.
\label{fig:end}}
\end{figure}
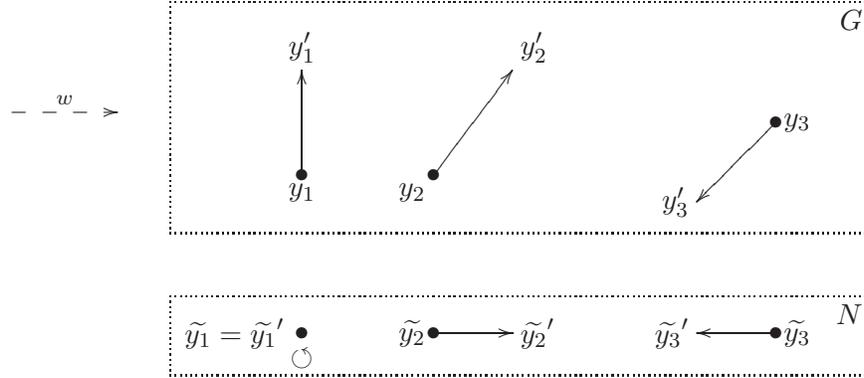


\begin{remark}\label{rmk:endo_polytope}
A second geometric interpretation of the strongly endotactic condition
is in terms of the reactant polytope~$Q$, which we recall is the
convex hull of $\source{\RR}$.  We say that a reaction $y \to y'$
\emph{points out of} a set $P$ if the line segment from $y$ to $y'$
intersects $P$ only at the point~$y$.  A network is strongly
endotactic if and only if (1)~no reaction with reactant on the
boundary of $Q$ points out of $Q$, and (2)~for all vectors $w$ that
are not orthogonal to the stoichiometic subspace, the $\leq_w$-maximal
face of $Q$ contains a reactant $y$ such that there exists a
nontrivial reaction $y \to y'$ that points out of the face (either
along the boundary of $Q$ or into the relative interior of $Q$).  See
Examples~\ref{ex:from_fig} and~\ref{ex:str_end}.
\end{remark}


\begin{definition}\label{d:endotactic}
Fix a reaction network $\GG$.
\begin{enumerate}
\item%
The network $\GG$ is $w$-\emph{endotactic} for some $w \in \bR^\SS$ if
$$%
  \< w,\flux{r} \> < 0
$$
for all $w$-essential reactions $r \in \RR_w$ such that $\source{r}
\in \supp w{\SS,\CC,\RR}$.

\item%
The network $\GG$ is \emph{$W$-endotactic} for a subset $W \subseteq
\bR^\SS$ if $\GG$ is $w$-endotactic for all vectors $w \in W$.

\item%
The network $\GG$ is \emph{endotactic} if it is $\bR^\SS$-endotactic.

\item\label{d:strong_endotactic}%
$\GG$ is \emph{strongly endotactic} if it is endotactic and for every
vector $w$ not orthogonal to the stoichiometric subspace of
$\GG$, there exists a reaction $r = (y \to y')$ in $\RR$ such that
\begin{enumerate}[(i)]
\item%
$y >_w y'$ (i.e., $\<w, \flux{r}\> < 0 $) and
\item%
$y$ is $\leq_w$-maximal among all reactants in $\GG$: $y \in
\init{w}{\source{\RR}}$.

\end{enumerate}
\end{enumerate}
\end{definition}


\begin{example}\label{ex:from_fig}
For the network $G$ in Figure~\ref{fig:end}, the reactant polytope $Q$
is the convex hull of the reactants $y_1, y_2, y_3$ (labeled by
$\bullet$), and both reactions $y_1 \to y_1'$ and $y_3 \to y_3'$ point
out of the triangle $Q$.  Thus $G$ is not strongly endotactic.
\end{example}

\begin{example}\label{ex:str_end}
The network
$$%
  0 \to 3A+B \quad \quad 2A \to B \quad \quad 2B \to A+B
$$
is strongly endotactic (thus, endotactic), but not weakly reversible.
In light of Remark~\ref{rmk:endo_polytope}, this can be seen from the
reaction diagram and reactant polytope~$Q$, which is the convex hull
of the reactants $0, 2A, 2B$ (marked by $\bullet$ in
Figure~\ref{f:endo}).
\begin{figure}[ht]
\begin{align*}
  \begin{xy}<14mm,0cm>:
  (0,0) 	="y1" *{\bullet} *+!R{0} ; 
  (3,1) 	="y1'" *+!L{3A+B}  *{}; 
  (2,0) 	="y2" *{\bullet} *+!L{2A}  *{}; 
  (0,1)	="y2'" *+!R{B}  *{}; 
  (0,2) 	="y3" *+!R{2B}  *{\bullet}; 
  (1,1) 	="y3'" *{} *+!DL{A+B}  *{}; 
  (.4, .5) = "Q" *{Q}; 
  {\ar "y1";"y1'"*{}  }; 		
  {\ar "y2";"y2'"*{}  }; 		
  {\ar "y3";"y3'"*{}  }; 		
  "y1";"y2" **\dir{.};	
  "y2";"y3" **\dir{.};	
  "y3";"y1" **\dir{.};	
  \end{xy}
\end{align*}
\caption{\label{f:endo}Reaction network from Example~\ref{ex:str_end}.}
\end{figure}
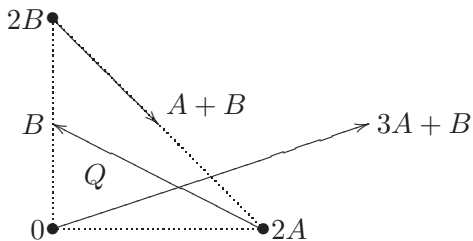
Indeed, no reaction points out of the triangle $Q$, and each proper
face of $Q$---an edge or a vertex---contains at least one reactant in
a reaction that points out of that face.
\end{example}


\begin{example}\label{ex:end'}
The network from Examples~\ref{ex:end_intro} and~\ref{ex:end}, whose
reactions are $2A \lra A+B$ and $B \to 0 \to 2B$, is endotactic but
not strongly endotactic.
\end{example}


The next lemma and the following two corollaries provide examples of
strongly endotactic reaction networks.  For notation, a set of
complexes of a network is a \emph{union of linkage classes} if it is
the set of complexes in a union of linkage classes of the reaction
graph.

\begin{lemma}\label{l:w-rev}
Let $\GG$ be a weakly reversible reaction network.  For a vector $w
\in \bR^\SS$, let $T_w = \init{w}{\source{\RR}}$ denote the set
of $\leq_{w}$-maximal reactants.
Assume that $w \in \bR^\SS$ is orthogonal to the stoichiometric
subspace of $\GG$ whenever $T_w$ is a union of linkage classes.  Then
$\GG$ is strongly endotactic.
\end{lemma}
\begin{proof}
Assume that $w \in \bR^\SS$ is not orthogonal to the stoichiometric
subspace of $\GG$.  Every weakly reversible network is
endotactic~\cite[Lemma 4.5]{CNP}, so it remains only to show that
there is a reaction going from $T_w$ to the complement $\CC \setminus
T_w$.  All complexes of a weakly reversible network are reactants, so
if $y \in T_w$, then $y \geq_w y'$ holds for all complexes $y'$ in
$\CC$.  Thus, by the hypothesis that $T_w$ is not a union of linkage
classes, there exists a reaction that goes from $T_w$ to the
complement $\CC \setminus T_w$ or from $\CC \setminus T_w$ to $T_w$.
In the former case, we are done; in the latter case, weak
reversibility of the network implies that there is some other reaction
that goes from $T_w$ to the complement.  Therefore $\GG$ is strongly
endotactic.
\end{proof}

We obtain the following corollary for weakly reversible networks.
\begin{corollary}\label{cor:same_subspace}
If each linkage class of a weakly reversible network $\GG$ has the
same stoichiometric subspace, namely that of $\GG$ itself, then $\GG$
is strongly endotactic.
\end{corollary}
\begin{proof}
Suppose each linkage class has the same stoichiometric subspace $H$.
For $w \in \bR^\SS$ not orthogonal to $H$, let $y$ be a
$\leq_w$-maximal complex, i.e., $y \in T_w = \init{w}{\source{\RR}}$.
By Lemma~\ref{l:w-rev}, it suffices to show that $T_w$ is not a
union of linkage classes.  Letting $G_j$ denote the linkage class of
$y$, it follows that $H$ is spanned by the vectors $z-y$, where $z$ is
a complex in $G_j$.  Thus there exists $z^*$ in $G_j$ such that $y >_w
z^*$, because otherwise $w$ would be orthogonal to $H$.  So, $z^*
\notin T_w$, which implies that $T_w$ is not a union of linkage
classes.
\end{proof}

Corollary~\ref{cor:same_subspace} implies that the networks Anderson
treated \cite{AndersonBd11, Anderson11} are strongly endotactic.

\begin{corollary}\label{cor:one_l_class}
Every weakly reversible reaction network with exactly one linkage
class is strongly endotactic.
\end{corollary}


\subsection{Mass-action differential inclusions}\label{s:m-a}

We now recall from~\cite{ProjArg} how a confined reaction system gives
rise to a mass-action differential inclusion.
In what follows, we assume that all manifolds have finite~dimension.


\begin{definition}\label{d:diff_incl}
Let $M$ be a smooth manifold with tangent bundle $\pi_M : TM \to M$.
A~\emph{differential inclusion} on $M$ is a subset $X \subseteq TM$.
\end{definition}


\begin{definition}
Let $X$ be a differential inclusion on a smooth manifold $M$.
\begin{enumerate}
\item%
Let $I\subseteq \Rnn$ be a nonempty interval (in particular,
connected) containing its left endpoint.  A differentiable curve $f: I
\to M$ is a \emph{trajectory} of $X$ if the tangent vectors to the
curve lie in~$X$.\vspace{-1ex}
\item%
An unbounded interval is a \emph{ray}.  A trajectory~$f$ defined on a
ray \emph{eventually} has property $P$ if there exists $T > 0$ such
that property $P$ holds for~$f$ whenever $t \geq T$.  The
\emph{$\omega$-limit set} of a trajectory~$f$ defined on a ray is the
set
\begin{equation*}
  \omega(f) = \big\{x \in \ol M \mid f(t_n) \to x \text{ for
  some sequence } t_n \in I \text{ with } t_n \to \infty \big\}
\end{equation*}
of accumulation points of~$f$ arising from a sequence of times tending
to infinity.
\end{enumerate}
\end{definition}

The next definition makes use of the notation $x^y = x_1^{y_1} \cdots
x_m^{y_m}$, for $x,y \in \bR^m$.


\begin{definition}\label{d:MADiffI}
The \emph{mass-action differential inclusion} of a confined reaction
system, given by a reaction network $\GG$ with tempering~$\kappa$ and
invariant polyhedron~$\invtPoly$, is the differential inclusion on
$\bR^\SS_{>0}$ in which the fiber over a point $x \in \relIntP$ is
$$
  \Big\{\sum_{r\in \RR} k_r x^{\source{r}}  \flux{r} \ \big|\ k_r\in
  \kappa(r) ~ \text{for all } r \in \RR \Big\} \subseteq \bR^\SS =
  T_x \Rplus^\SS,
$$
and the fiber over all other points $x \in \Rplus^\SS \setminus
\relIntP$ is empty.
\end{definition}

\begin{example}\label{ex:m-a}
One possible tempering on the network from Example~\ref{ex:str_end} is
given by $\kappa(0 \to 3A+B) = [1,2]$, and $\kappa(2A \to B) = \{3\}$,
and $\kappa(2B \to A+B) = [4,5]$.  Every trajectory $x(t) =
\big(x_A(t),x_B(t)\big)$ of the resulting mass-action differential
inclusion
satisfies
\begin{align*}
  \dot{x}_A &= 3 \cdot k_1(t) - 6 \cdot x_A(t)^2 + k_3(t) \cdot x_B(t)^2
\\\dot{x}_B &= k_1(t) + 3 \cdot x_A(t)^2 - k_3(t) \cdot x_B(t)^2 ~,
\end{align*}
where $k_1(t) \in [1,2]$ and $k_3(t) \in [4,5]$ for all time $t$.	
\end{example}


\section{Conjectures related to persistence and permanence}\label{s:conjs}

In this section, we recall several conjectures related to the
persistence of reaction networks.  First, some definitions
from~\cite[\S2]{ProjArg} are required; see \cite[\S2]{ProjArg} for
remarks on relations between these various concepts and comparisons
with similar notions in the literature.

\begin{definition}\label{d:persistent_permanent}
Let $\ol M$ be a smooth manifold with corners whose interior is $M =
\ol M \setminus \partial \ol M$, and let $V \subseteq \partial \ol M$
be a subset of the boundary.  A differential inclusion $X \subseteq
TM$ is
\begin{enumerate}
\item%
\emph{persistent} if the closure in~$\ol M$ of every
trajectory of~$X$ is disjoint from~$\partial \ol M$.
\item%
\emph{repelled by $V$} if for every open set $O_1 \subseteq \ol M$
with $\ol V \subseteq O_1$, there exists a smaller open set
$O_2\subseteq O_1$ with $\ol V \subseteq O_2$ such that for every
trajectory $f: I \to M$ of~$X$, if $f(\inf{I}) \notin O_1$ then $f(I)
\cap O_2$ is empty; in other words, if the trajectory begins outside
of~$O_1$, then the trajectory never enters $O_2$.
\end{enumerate}
If $\ol M$ is compact, then a differential inclusion $X \subseteq TM$
is \emph{permanent} if it is persistent and there is a compact subset
$\Omega \subseteq M$ such that for every ray $I$, every trajectory of
$X$ defined on $I$ is eventually contained in $\Omega$.
\end{definition}



\begin{definition}\label{d:strongly_persistent}
A confined reaction system $N$, specified by a reaction network $\GG$
together with a tempering $\kappa$ and an invariant
polyhedron~$\invtPoly$, is \emph{persistent} (respectively,
\emph{permanent}) if the mass-action differential
inclusion
on $M = \Rplus^\SS$ arising from the reaction system $N$ is persistent
(respectively, permanent) when viewed with respect to the
compactification $\ol M = [0,\infty]^\SS$.  More generally, a network
$\net$ is itself \emph{persistent} or \emph{permanent}
if for all choices of temperings $\kappa$ and invariant polyhedra
$\invtPoly$, the resulting mass-action differential inclusion
has the corresponding property.
\end{definition}

We now state three conjectures, in increasing level of strength, and
then state a related fourth conjecture.  The first conjecture concerns
so-called ``complex-balanced'' systems, which form a well-studied
subclass of weakly reversible mass-action ODE systems.  Moreover, this
class contains all so-called ``detailed-balanced'' systems and weakly
reversible ``deficiency zero'' systems.  Many properties of
complex-balanced systems (as well as detailed-balanced systems and
deficiency zero systems) were elucidated by Feinberg, Horn, and
Jackson in the 1970s, and we provide only an overview here.  (A
definition of complex-balanced systems can be found in any of the
following references: \cite{TDS,Feinberg72,Horn72,HornJackson}.)  For
such systems, it is known that a unique steady state resides within
the interior of each invariant polyhedron~$\invtPoly$.  This steady
state is called the Birch point in~\cite{TDS} due to the connection to
Birch's Theorem (see Section~\ref{s:extBirch}).
Moreover, a strict Lyapunov function exists for this Birch point, so
local asymptotic stability relative to $\invtPoly$ is
guaranteed~\cite{HornJackson}.  An open question is whether all
trajectories with initial condition interior to $\invtPoly$ converge
to the unique Birch point of $\invtPoly$.  The assertion that the
answer is ``yes'' is the content of the following conjecture, which
was stated first by Horn in 1974~\cite{Horn74} and was given the name
``Global Attractor Conjecture'' by Craciun et al.~\cite{TDS}.


\begin{conjecture}[Global Attractor Conjecture]\label{conj:GAC}
For every invariant polyhedron $\invtPoly$ of every complex-balanced
system,
the Birch point $\ol x \in \invtPoly$
is a global attractor of $\relIntP$.
\end{conjecture}
Due to the strict Lyapunov function, the \GAC is equivalent to the
following:\ \emph{every complex-balanced system is persistent}.  This
suggests the following more general conjecture, which was first stated
by Feinberg in 1987~\cite[Remark 6.1E]{Fein87}.


\begin{conjecture}[Persistence conjecture]\label{conj:persistence}
Every weakly reversible mass-action kinetics ODE system is persistent.
\end{conjecture}

Conjecture~\ref{conj:persistence} was generalized recently by Craciun,
Nazarov, and Pantea in the following three ways: the weakly reversible
hypothesis is weakened to endotactic, fixed reaction rate constants
are allowed to vary within bounded intervals (i.e., they are
tempered), and the conclusion of persistence is strengthened to
permanence \cite[\S 4]{CNP}.



\begin{conjecture}[Extended permanence conjecture]\label{conj:ext_permanence}
Every endotactic reaction network is permanent.
\end{conjecture}

\begin{remark}
Conjecture~\ref{conj:ext_permanence} captures the intuitively
appealing idea that if a reaction diagram ``points inwards'' then the
corresponding dynamics in concentration space must also roughly
``point inwards.'' We interpret this conjecture as a suggestion that
the geometry of the reaction diagram ought to be viewed literally as a
combinatorial representation of the dynamics.  From this perspective,
Conjecture~\ref{conj:ext_permanence} is a first step in a research
program to complete the details of this correspondence.  In our
previous work~\cite[Question~5.26]{ProjArg}, we suggested a framework
within which additional aspects of this correspondence might be
explored.  In particular, we utilized the standard dynamical system
notion of topological equivalence to ask if the qualitative nature of
the dynamics remains invariant under reasonable transformations to the
reaction diagram.  This idea was most pithily expressed by asking for
the richest domain category (of reaction diagrams) from which the
mass-action differential inclusion remains a functor \cite{ProjArg}.
\end{remark}

The following fourth conjecture was stated recently by
Anderson~\cite[\S 1.1]{AndersonBd11}.  It would follow from
Conjecture~\ref{conj:persistence} according to our definition of
persistence for reaction networks
(which differs from some other definitions, cf.\
\cite[Remark~2.10]{ProjArg}).


\begin{conjecture}[Boundedness conjecture]\label{conj:boundedness}
Every weakly reversible mass-action kinetics ODE system has bounded
trajectories.
\end{conjecture}


Although all four conjectures remain open, some progress has been made
in recent years.  Conjecture~\ref{conj:boundedness} is true for
complex-balanced systems (due to the Lyapunov function:
see~\cite[Lemma~3.5]{SM} for details) and was resolved recently for
mass-action ODE systems with only one linkage class (under some
additional mild hypotheses) by Anderson \cite{AndersonBd11}.
Conjecture~\ref{conj:GAC} has been proved for systems of dimension at
most three \cite{Anderson08,AndersonShiu10,TDS,CNP,Pantea} and also
when the network contains only one linkage class~\cite{Anderson11}.
Conjecture~\ref{conj:ext_permanence}, and thus
Conjectures~\ref{conj:GAC} and~\ref{conj:persistence} as well, has
been resolved in dimensions at most~2 by Craciun, Nazarov, and Pantea
for the systems for which Conjecture~\ref{conj:boundedness}
holds~\cite{CNP,Pantea}.  These results are due in part to the
analysis of steady states on the boundary of invariant polyhedra.

Indeed, it is known that in the case of complex-balanced systems, the
$\omega$-limit set is contained in the set of steady states.
In this setting, it has been proved that certain boundary steady
states are not $\omega$-limit points.  These include vertices of an
invariant polyhedron, according to Anderson and Craciun et
al.~\cite{Anderson08,TDS}; interior points of facets, according to
Anderson and Shiu~\cite{AndersonShiu10}; and interior points of
``weakly dynamically non-emptiable'' faces, according to Johnston and
Siegel~\cite{JohnstonSiegel}.  In fact, trajectories are repelled by
such points.  Also, some networks have no boundary steady states; for
example, Shinar and Feinberg have proved that weakly reversible
``concordant'' networks have this property~\cite{concordant}.
Additionally, the three-dimensional case was resolved by
Pantea~\cite{Pantea}.  The remaining cases for
Conjecture~\ref{conj:GAC} are systems of dimension 4 and higher in
which steady states lie on faces of dimension at least~1 and
codimension at least~2.

There have been several other recent approaches to persistence-type
results.  First, Siegel and Johnston proved that for a
complex-balanced system, the positive orthant can be subdivided into
strata in which trajectories must obey certain linear Lyapunov
functions~\cite{JohnstonSiegel_stratum}.  Second, Angeli, De Leenheer,
and Sontag gave persistence criteria that allow for differential
inclusions and time-varying rate constants (i.e., with
temperings)~\cite{ADS09,ADS11}.  For instance they proved that
$\omega$-limit points arising from reaction systems must lie in the
relative interior of faces of the positive orthant defined by
``critical siphons''; therefore, networks without critical siphons are
persistent.  (A~siphon is a subset of the species whose absence is
forward-invariant with respect to the dynamics;
see~\cite{ADS11,ShiuSturmfels} for a precise definition of (critical)
siphon.)  Another approach makes use of the theory of monotone
systems; for instance, see the recent works of Angeli, De Leenheer,
and Sontag~\cite{ADS10}, Banaji and Mierczynski~\cite{BanajiM}, and
Donnell and Banaji~\cite{DB}.  Related work has also used the theory
of Petri nets~\cite{Angeli08,ADS11}.  Additionally, Gopalkrishnan
proved that every network that violates
Conjecture~\ref{conj:persistence} must exhibit a certain catalytic
property~\cite{cat}.  Finally, we refer the reader to work by Gnacadja
\cite{Gnacadja08, GnacadjaPers} that considered a stronger version of
persistence, called ``vacuous persistence'', which allows for
trajectories with initial condition on the boundary of an invariant
polyhedron as well as in the relative interior; he showed that certain
enzymatic networks are persistent in this stronger sense.

For mass-action ODE systems in which persistence is difficult to prove
directly, it is possible that the system is dynamically equivalent
to---that is, gives rise to the same ODE system as---one that is more
easily seen to be persistent.  To this end, Szederk\'enyi and Hangos
gave a method for determining whether a given system is dynamically
equivalent to a complex-balanced or a detailed-balanced one~\cite{SH}.
Similarly, Johnston and Siegel gave algorithms that determine whether
a given system is dynamically equivalent---or more generally, is
linearly conjugate---to one from certain classes (such as weakly
reversible systems)~\cite{JS11,JSS}.


\enlargethispage{.3ex}
The results in the current work are complementary to those described
above: in Section~\ref{s:results} we resolve
Conjectures~\ref{conj:GAC}--\ref{conj:boundedness} for strongly
endotactic networks in the general setting of mass-action differential
inclusions.

\begin{remark}
The class of complex-balanced systems can be extended in a
straightforward way to include some power-law systems, by relaxing the
requirement of a complex-balanced system that its complexes be
nonnegative integer vectors.  In this larger setting, our results
concerning the \GAC remain valid.
\end{remark}


We conclude this section by noting that non-endotactic networks can
fail to be persistent; for instance, the network $A \to B$.

\begin{example}\label{ex:not-persistent}
A less trivial example of a non-endotactic non-persistent network is
\begin{align*}
S_0 + F \overset{\kappa_1}\too S_1 + E \overset{\kappa_2}\too S_2 + F
\qquad
S_2 \overset{\kappa_3}\too S_1\overset{\kappa_4}\too  S_0.
\end{align*}
Write each concentration vector as $x = (x_{S_0},x_{S_1}, x_{S_2},
x_{E}, x_F)$.  In \cite[\S IV.A]{Angeli08}, Angeli showed that if the
reaction rate constants satisfy $\kappa_2 < \kappa_4$, then for the
resulting mass-action ODE system, the boundary steady state $x^* =
(1,0,0,1,0)$ is locally asymptotically stable relative to its
invariant polyhedron.  Hence this non-endotactic network is not
persistent.
\end{example}


\section{Extensions of Birch's theorem}\label{s:extBirch}

This section extends Birch's theorem in order to deduce
Corollary~\ref{cor:bddStoicSubComponent}, which later is used together
with Birch's theorem to prove one of our key results,
Theorem~\ref{thm:LyFncWorks_cpc}.  The reader may wish to skip this
section on the first pass.  The following notation is employed in
this~section.
\begin{notation}\label{not:products}
For vectors $\alpha, \beta \in \Rplus^m$, a constant $\theta \in
\Rplus$, a positive vector $\Theta = (\Theta_1,\ldots,\Theta_m) \in
\Rplus^m$, and $w = (w_1,\ldots,w_m) \in \bR^m$, set
\begin{align*}
\alpha*\beta &= (\alpha_1\beta_1, \ldots, \alpha_m\beta_m) \in\Rplus^m
\\
\alpha/\beta &= (\alpha_1/\beta_1, \ldots, \alpha_m/\beta_m) \in\Rplus^m
\\
\theta^w &= (\theta^{w_1}, \ldots, \theta^{w_m}) \in \Rplus^m
\\
\Theta^w &= \Theta_1^{w_1} \cdots \Theta_m^{w_m} \in \Rplus
\\
\log(\Theta) &= (\log \Theta_1,\ldots,\log \Theta_m).
\end{align*}
Also, $S^{m-1}$ denotes the unit sphere of dimension $m-1$ in $\bR^m$.
\end{notation}

\begin{remark}
The top two displayed notations express the group operations in the
multiplicative group (``positive real algebraic torus'') $\Rplus^m$,
where $\alpha/\beta = \alpha*\beta^{-1}$.
\end{remark}

\subsection{Toric rays}\label{sub:toric rays}

\begin{definition}\label{d:toric_ray}
Let $\alpha\in\Rplus^m$.
\begin{enumerate}
\item%
Let $w\in S^{m-1}$.  The \emph{toric ray} from $\alpha$ in direction
$w$ is the curve
$$%
  \Toric_\alpha(w)= \{ \alpha * \theta^w \mid 1 \leq \theta <
  +\infty\} \subseteq \Rplus^m.
$$
\item%
Let $N\subseteq S^{m-1}$.  The \emph{toric pencil}
$\Toric_\alpha(N)$ from $\alpha$ in directions $N$ is the union of
toric rays in directions $w\in N$.
\end{enumerate}
\end{definition}

\begin{remark}\label{rmk:ray}
The toric ray $\Toric_\alpha(w)$ is the coordinatewise exponential
of the infinite ray in $\bR^m$ with direction~$w$ that originates at
the point $\log(\alpha)$.
\end{remark}

\begin{remark}\label{rmk:polar_c}
For a given $\alpha \in \Rplus^m$, the open rays $\{\log\alpha + (\log
\theta) w \mid 1 <\theta < +\infty\}$ partition the punctured space
$\bR^m \setminus \{\log \alpha\} $.  Indeed, when $\alpha$ is the
multiplicative identity vector $(1,\ldots,1)$, so that $\log \alpha$
is the origin, the coordinates $w \in S^{m-1}$ and $\log \theta \geq
0$ may be viewed as polar coordinates on~$\bR^m$.
\end{remark}

\begin{definition}\label{d:Lyap}
For $\alpha \in \Rplus^m$, the function $g_\alpha: \bR_{\geq 0}^m \to
\bR$ is given by
$$%
  (x_1, \ldots, x_m)
  \longmapsto
  \sum_{i=1}^m x_i \log \frac{x_i}{\alpha_i} - x_i,
$$
where $0\log 0$ is defined to be~$0$.  Because $\lim\limits_{x_i \to
0^+} x_i \log x_i = 0$, this definition is the unique choice that
makes the function $g_\alpha$ continuous on the closed nonnegative
orthant $\bR_{\geq 0}^m$, including the boundary.
\end{definition}

The next lemma shows that the function $g_\alpha$ captures the
geometry of toric rays.

\begin{lemma}\label{l:toricgrad}
Let $\alpha \in \Rplus^m$, and let $w \in \bR^m$ be a unit vector.
For every point $x \in \Toric_\alpha(w)$, either $x = \alpha$ or
$g_\alpha(x)$ is nonzero and has gradient direction~$w$; that is,
$\frac{\nabla g_\alpha(x)}{\lVert \nabla g_\alpha(x) \rVert} = w$.
\end{lemma}

In words, the lemma says that along a toric ray from $\alpha$, the
direction of the gradient of~$g_\alpha$ matches the direction of the
toric ray.

\begin{proof}
Let $\alpha$ and $w$ be as in the statement of the lemma, and let
$\alpha \neq x\in\Toric_\alpha(w)$, so $x = \alpha*\theta^w$ for
some $\theta \in (1,\infty)$.  A straightforward calculation shows
that
\begin{align}\label{eq:grad_toric}
\nabla g_\alpha (x) = \log (x/\alpha)  =  (\log \theta) w.
\end{align}
Since $\theta\in (1,\infty)$, the coefficient satisfies $\log\theta >
0$.  Thus, the gradient of $g_\alpha$ at $x$ is in direction~$w$.
Since $w$ is a unit vector, the normalized gradient is $\frac{\nabla
g_\alpha(x)}{\lVert \nabla g_\alpha(x) \rVert} = w$.
\end{proof}

\begin{remark}
The function $g_\alpha$ has a uniqueness property with respect to
toric rays: if $f$ is a function on $\Rplus^m$ whose gradient along
every toric ray from $\alpha$ points in the direction of that toric
ray, then the level sets of $f$ and $g_\alpha$ form the same
foliation.
\end{remark}

\subsection{Outer normal cones}\label{sub:outer}


\begin{definition}\label{d:onc}
Fix a point $x$ of the compactification $[0,\infty]^m$ of $\Rplus^m$.
Denote by
$$%
  \Sigma_0 = \big\{i \in \{1,\ldots,m\} \mid x_i = 0 \big\}
  \quad\text{and}\quad
  \Sigma_\infty = \big\{i \in \{1,\ldots,m\} \mid x_i = \infty\big\}
$$
the sets of zero and infinite coordinates of~$x$, respectively.
$[0,\infty]^m$ has \emph{outer normal cone}
$$%
C_x[0,\infty]^m = \{u \in \bR^m \mid
  \op{supp}(u^-) \subseteq \Sigma_0
  \text{ and }
  \op{supp}(u^+) \subseteq \Sigma_\infty\}
$$
at $x$, where
$u = u^+ - u^-$
writes $u$ as a difference of nonnegative vectors with disjoint
support.
\end{definition}


The following lemma states that if a sequence $x(n)$ in $\Rplus^m$
converges to a boundary point $x^* \in \partial \cpcOrth$, then the
limiting direction of $\nabla g_\alpha (x(n))$, if it exists, must
lie in the outer normal cone of the hypercube $[0,\infty]^m$ at
$x^*$.  It also states a converse.


\begin{lemma}\label{l:limDirs}
Fix a point $x^*$ in the compactification $[0, \infty]^m$ of
$\Rplus^m$.  Let $\alpha \in \Rplus^m$ be any positive vector.  A unit
vector $u \in \bR^m$ lies in the outer normal cone of $[0, \infty]^m$
at~$x^*$ if and only if there exists a sequence $x(n)$ in~$\Rplus^m$
such that
\begin{align}\label{eq:lim_grad}
  x(n) \to x^*
  \quad\text{and}\quad
  \frac{\nabla g_\alpha\big(x(n)\big)}
       {\lVert\nabla g_\alpha\big(x(n)\big)\rVert} \to u.
\end{align}
\end{lemma}
\begin{proof}
First suppose that the sequence $x(n)$ for $n \in \Zplus$
satisfies~\eqref{eq:lim_grad}.  Since $\nabla g_\alpha(x(n)) =
\log(x(n)/\alpha)$ by~(\ref{eq:grad_toric}), the $i$th coordinate of
$\nabla g_\alpha(x(n))$ goes to $-\infty$ whenever $x^*_i=0$.
Similarly, the $i$th coordinate of $\nabla g_\alpha(x(n))$ goes to
$\infty$ whenever $x^*_i = \infty$.  All other coordinates go to a
finite limit, namely $\log\big(x^*_i/\alpha_i\big))$.  Therefore, the
limit $u$ of the sequence $\frac{\nabla g_\alpha (x(n)) }{ \lVert
\nabla g_\alpha (x(n)) \rVert}$ must lie in the outer normal cone
at~$x^*$.

For the converse, let $u \in \bR^m$ be a unit vector in the outer
normal cone of $[0, \infty]^m$ at~$x^*$.  Define the point $\beta \in
\Rplus^m$ with coordinates
$$
\beta_i =
  \begin{cases}
  x^*_i &\text{if } x^*_i\in (0,\infty)
  \\
  1 &\text{if } x^*_i = 0 \text{ or } x^*_i = \infty.
  \end{cases}
$$
Fix a sequence $\thetaN$ in $\bR_{>1}$ with limit $+\infty$.  Consider
the sequence of points $x(n)=\beta * \thetaN^u$ along the toric ray
from $\beta$ in direction~$u$.
For $i = 1,\ldots,m$,
\begin{align*}
  x(n)_i =
  \begin{cases}
    x_i^*\cdot\thetaN^{u_i}&\text{if } 0<x_i^*< \infty\text{ (thus } u_i = 0)
  \\1 \cdot \thetaN^{u_i} &\text{if } x_i^* = 0 \text{ (thus } u_i < 0)
  \\1 \cdot \thetaN^{u_i} &\text{if } x_i^* = \infty\text{ (thus } u_i > 0).
  \end{cases}
\end{align*}
Hence, in the first case above, $x(n)_i = x_i^*$ for all~$n$; in the
second case, $x(n)_i \to 0 = x_i^*$; and in the third case, $x(n)_i
\to \infty =x_i^*$.  So, to show~\eqref{eq:lim_grad}, it remains only
to prove
that the normalized gradient of $g_\alpha$ at $x(n)$ converges to $u$.
It is straightforward to verify that
$$%
  \nabla g_\alpha\big(x(n)\big) = \big(\log \thetaN\big) u + c^*,
$$
where $c^*$ is the vector with coordinates
$c_i = \log(x^*_i/\alpha_i)$ if $x^*_i\in (0,\infty)$ and 0 if $x^*_i
= 0$ or $x^*_i = \infty$.  Thus, as $\log \thetaN$ goes to $\infty$,
the direction of $\nabla g_\alpha (x(n))$ converges to $u$.
\end{proof}

\begin{lemma}\label{l:oncint}
Let $H\subseteq\bR^m$ be a linear subspace.  Let $p\in\Rplus^m$.
Consider the polyhedron $\invtPoly = (p + H) \cap \bR^m_{\geq 0}$.
Let $x^* \in \partial[0,\infty]^m$.  If $x^* \in \ol\invtPoly$, where
the closure is taken in $[0,\infty]^m$, then the outer normal cone of
$[0,\infty]^m$ at $x^*$ meets $H^\perp$
only at the origin.
\end{lemma}

The key geometric insight for the proof of Lemma~\ref{l:oncint}
is that if a translate~$L'$ of a support hyperplane $L$ intersects the
interior of a polytope~$P$, then $L'$ fails to intersect the support
face $L \cap P$ of the untranslated hyperplane, since $\<u, L'\> \neq
\<u, L \cap P\>$ for any $u \in L^\perp$.

\begin{proof}
Take a sequence $x(n)$ in $\relIntP$ with $\lim\limits_{n \to \infty}
x(n) = x^*$, where the limit is taken in the compactification
$[0,\infty]^m$,
and suppose that $u$ is a unit vector in $C_{x^*}[0,\infty]^m$.  At
least one of $u^+$ and $u^-$ is nonzero.  If $u^+ \neq 0$, then
$$%
  \<u, x(n)\>
  \ =\
  \<u^+, x(n)\> - \<u^-, x(n)\>
  \ \geq\
  \<u^+, x(n)\>
  \ \to\
  \infty
$$
as $n \to \infty$, because the support of~$u^+$ is contained in the
set of $\infty$-coordinates $\Sigma_\infty$ of $x^*$ and $x(n) \to
x^*$.
But if $v \in H^\perp$ then $\<v, x\>$ takes a constant finite value
for all $x \in \invtPoly$; therefore $u \in C_{x^*}[0,\infty]^m
\implies u \not\in H^\perp$.  In the remaining case, when $u^+ = 0$
and $u^- \neq 0$,
$$%
  \<u, x(n)\>
  \ =\
  -\<u^-, x(n)\>
  \ \to\
  0
$$
as $n \to \infty$, because $\op{supp}(u^-) \subseteq \Sigma_0$.  This
also prevents $u \in H^\perp$: the inner product $\<v, x(n)\>$
maintains a constant value for all $n$ when $v \in H^\perp$, whereas
$\<u, x(n)\>$ is strictly negative and increasing toward~$0$, because
$x(n)$ is a strictly positive vector.
\end{proof}

\subsection{Birch's theorem}\label{sub:birch}

In Section~\ref{s:extensions}, we prove two extensions to the
following theorem.

\begin{theorem}[Birch's theorem]\label{t:birch}
Let $H\subseteq\bR^m$ be a linear subspace, and let $\alpha,p
\in\bR^m_{>0}$.  The relative interior of the polyhedron $\invtPoly =
(p+H)\cap\bR^m_{\geq 0}$ intersects the toric pencil
$\Toric_\alpha(\HPerp)$ at exactly one point.
\end{theorem}

Variants of Birch's theorem appear across the mathematical sciences.
In algebraic statistics, toric pencils
appear as log-linear statistical models, and in this setting,
Theorem~\ref{t:birch} was first proved by Birch in 1963~\cite{Birch}.
In dynamical systems, in the setting of chemical reaction systems, the
theorem was proved by Horn and Jackson~\cite[Lemma~4B]{HornJackson}.

In the setting of quasi-thermostatic chemical reaction systems (a
class introduced by Horn and Jackson that includes complex-balanced
systems and therefore deficiency zero systems as
well~\cite{Fein87,HornJackson}), the toric pencil
in Birch's theorem is equal to the set of positive steady states.  In
that context, the toric pencil
usually is written as $\{c \in \Rplus^m \mid \log \frac{c}\alpha \in
\HPerp\}$, where $\alpha$ is a given steady state.

Theorem~\ref{t:birch} prompts the following definition.

\begin{definition}
Let $H\subseteq\bR^m$ be a subspace, and let $\alpha,p \in\bR^m_{>0}$.
The $\alpha$-\emph{Birch point} of the polyhedron $\invtPoly =
(p+H)\cap\bR^m_{\geq 0}$ is the unique point in the intersection
$\op{int}\invtPoly\cap\Toric_\alpha(\HPerp)$.
\end{definition}
\noindent
We refer to the \emph{Birch point} when the choice of $\alpha$ and
$\invtPoly$ is clear from context.  Horn~\cite{Horn74} conjectured
that every Birch point of a complex-balanced system is a global
attractor of the corresponding invariant polyhedron $\op{int}\invtPoly$
(Conjecture~\ref{conj:GAC}).

The following is a conventional form of Birch's theorem in the setting
of algebraic statistics, as stated in the book by Pachter and
Sturmfels~\cite[Theorem 1.10]{ASCB}.

\begin{theorem}[Birch's theorem, restated]\label{t:birchRestated}
Fix a $d \times m$ real matrix $A$ with columns denoted by $a_1,
\ldots, a_m$ and positive vectors $\alpha, p \in\bR^m_{>0}$.  The
image of the monomial map
\begin{align*}
  f_{A,\alpha}: \Rplus^d &\to \Rplus^m
\\                \Theta &\mapsto\alpha*(\Theta^{a_1},\ldots,\Theta^{a_m})
\end{align*}
intersects the polyhedron $\invtPoly = \{q \in \Rnn^m \mid Aq = A p\}$
in a single point.
\end{theorem}

\begin{remark}\label{rmk:twoBThmEquiv}
To see why Theorems~\ref{t:birch} and~\ref{t:birchRestated} are
equivalent, first observe that the two definitions of the polyhedron
$\invtPoly$ coincide because $H = \ker A$.  It remains to show that
\begin{align}\label{eq:toricRaySet}
f_{A, \alpha} ( \Rplus^d )
  = \Toric_\alpha ( \HPerp )
  = \{\alpha * \theta^w \mid \theta \in \Rplus \text{ and } w \in H^\perp\}.
\end{align}
Given a vector $\alpha*(\Theta^{a_1}, \ldots, \Theta^{a_m})$ in the
image of $f_{A,\alpha}$, when $\theta = \Theta_1$ and
$$
  w = a(1)+\big(\log_{\Theta_1}(\Theta_2)\big) a(2)+ \cdots +
  \big(\log_{\Theta_1}(\Theta_d)\big) a(d),
$$
where $a(j)$ denotes the $j$th row of~$A$, the point
$\alpha*(\Theta^{a_1}, \ldots, \Theta^{a_m}) = \alpha * \theta^w$ lies
in the set on the right-hand side of \eqref{eq:toricRaySet}.
Conversely, given $\alpha * \theta^w$ where $\theta >0$ and $w$ is a
linear combination of the rows of $A$: $w = c_1a(1) + \cdots + c_d
a(d)$, then for $\Theta = \big(\theta^{c_1}, \ldots,
\theta^{c_d}\big)$, it follows that $\alpha*\theta^w =
f_{A,\alpha}(\Theta)$ lies in the image of the map $f_{A,\alpha}$.
Thus \eqref{eq:toricRaySet}~holds.
\end{remark}

\subsection{Extensions}\label{s:extensions}

We extend Birch's theorem to the compactification $\cpcOrth$ of the
nonnegative orthant $\Rnn^m$.

\begin{theorem}[Birch's theorem at infinity]\label{t:birchExt1}
Fix a linear subspace $H\subseteq\bR^m$ and $\alpha,p \in\bR^m_{>0}$.
The $\alpha$-Birch point of the polyhedron $\invtPoly =
(p+H)\cap\bR^m_{\geq 0}$ is the unique point in the intersection
\begin{align}\label{eq:ext_B_intersection}
  \ol\invtPoly \cap \ol{\Toric_\alpha(H^\perp)},
\end{align}
where the two closures are taken in the compactification
$[0,\infty]^m$.
\end{theorem}

We call this theorem ``Birch's theorem at infinity'' because we are
studying the same intersection problem as in the original Birch's
theorem, but now the parameter $\theta$ along toric rays is allowed to
take the value $+\infty$.

\begin{proof}
The intersection \eqref{eq:ext_B_intersection} has just one point in
$\bR^m_{>0}$, namely the $\alpha$-Birch point of~$\invtPoly$; this
follows from the usual Birch's theorem (Theorem~\ref{t:birch}).
Therefore, we need only prove the lack of boundary points $x^*$ in the
intersection~\eqref{eq:ext_B_intersection}.
To this end, let $x^* \in \ol{\Toric_\alpha(\HPerp)} \cap
\partial \ol\invtPoly$ and choose a sequence $x(n) = \alpha *
\thetaN^{\wN} \to x^*$, where $\thetaN \to \infty$ and $\wN \in \HPerp
\cap S^{m-1}$.  Taking a subsequence if necessary, assume that $w(n)$
converges to some $u \in \HPerp \cap \sphere$.  By
Lemma~\ref{l:toricgrad}, $\frac{\nabla g_\alpha (x(n))}{\lVert
\nabla g_\alpha (x(n)) \rVert} = w(n) \to u$.  By
Lemma~\ref{l:limDirs}, the direction $u$ lies in the outer normal
cone of $\cube$ at~$x^*$.  Therefore the outer normal cone intersects
$\HPerp$ nontrivially, so by Lemma~\ref{l:oncint}, the point $x^*$
does not lie in $\partial \ol\invtPoly$.  Thus the
intersection~\eqref{eq:ext_B_intersection} is empty.
\end{proof}

By Theorem~\ref{t:birchExt1}, when $w \in \HPerp$ the toric ray
$\Toric_\alpha(w)$ either intersects $\ol\invtPoly$ at the Birch point
or does not approach $\ol\invtPoly$, including its boundary.  The
following theorem considers toric rays where $w$ lies more generally
in a neighborhood of~$\HPerp$ and states that these toric rays do not
approach $\ol\invtPoly$ outside a compact neighborhood of the Birch
point.


\begin{theorem}[Perturbed Birch's theorem]\label{thm:extBirch}
Fix a linear subspace $H\subseteq\bR^m$ and vectors $\alpha,p
\in\bR^m_{>0}$.  Let $q$ be the $\alpha$-Birch point in the polyhedron
$\invtPoly = (p+H)\cap\bR^m_{\geq 0}$, namely the unique point in the
intersection $\relIntP \cap \Toric_\alpha(H^\perp)$.  For every
relatively open neighborhood $\Oset$ of $q$ in the relative interior
of $\invtPoly$, the set $H^\perp \cap S^{m-1}$ of unit directions
along $H^\perp$ has a relatively open neighborhood $N$ in $S^{m-1}$
such that
$$%
  (\ol\invtPoly \setminus \Oset) \cap \ol{\Toric_\alpha(N)} =
  \emptyset,
$$
where the two closures are taken in the compactification
$[0,\infty]^m$.
\end{theorem}
\begin{proof}
For positive integers $n$, let $\left\{ N_n \right\}$ denote a
shrinking sequence of $\varepsilon_n$-neighborhoods of $H^\perp \cap
S^{m-1}$ in $S^{m-1}$ with radius $\varepsilon_n >0$ tending to 0.
(For instance, take $\varepsilon_n = 1/n$.)  It suffices to show that
the set
$$%
  \mathcal{Q}_n
  =
  (\ol\invtPoly \setminus \Oset) \cap \ol{\Toric_\alpha(N_n)}
$$
is empty for some~$n$.  The intersection of the sets $\mathcal{Q}_n$
for all~$n$ is $(\ol\invtPoly \setminus \Oset) \cap
\ol{\Toric_\alpha(H)}$, which is empty by Theorem~\ref{t:birchExt1}.
Since the sets $\mathcal{Q}_n$ are nested closed subsets of the
compact space $[0,\infty]^m$, a standard theorem of topology (see
Theorem~26.9 and the ensuing discussion in the book by
Munkres~\cite{Munkres}) implies that $\mathcal{Q}_n$ is empty for
large~$n$.
\end{proof}

Theorem~\ref{thm:extBirch} implies that if a toric ray
$\Toric_\alpha(w)$ intersects $\invtPoly$ outside of a neighborhood of
the Birch point $q$, then $w$ lies outside a neighborhood of $\HPerp$,
so the $\HPerp$-component of $w$ is not too dominant.  We quantify
this via a lower bound on the $H$-component.
For notation, every vector $w \in \bR^m$ is uniquely expressible as
the sum $w_H + w_{\HPerp}$ of a vector $w_H \in H$ and a vector
$w_{\HPerp} \in \HPerp$.


\begin{corollary}\label{cor:bddStoicSubComponent}
Fix a linear subspace $H\subseteq\bR^m$ and $\alpha,p \in\bR^m_{>0}$.
For every open subset $\Oset$ of the interior of the polyhedron
$\invtPoly = (p+H)\cap\bR^m_{\geq 0}$ such that $\Oset$ contains the
$\alpha$-Birch point, there exists
$\mu > 0$ such that
$\lVert w_H \rVert \geq \mu$ for all unit vectors $w$ satisfying
$\alpha * \theta^w \in \relIntP \setminus \Oset$ for some
\mbox{$\theta > 1$}.
\end{corollary}
\begin{proof}
By the proof of Theorem~\ref{thm:extBirch}, there exists $\varepsilon
> 0$ such that the $\varepsilon$-neighborhood $N$ of $H^\perp \cap
S^{m-1}$ in $S^{m-1}$ satisfies $(\ol\invtPoly \setminus \Oset) \cap
\ol{\Toric_\alpha(N)} = \emptyset$.
Now define
$$%
  \mu = \inf\big\{\lVert w_H \rVert \,\big|\, w \in S^{m-1} \setminus N\big\}.
$$
Then $\mu > 0$ because it is the infimum of a nonnegative continuous
function on a compact set that never takes the value~$0$ because
$w_H = 0 \implies w \in \HPerp \subseteq N$.  If
$\alpha * \theta^w \in \relIntP \setminus \Oset$,
then $w \notin N$ by
construction of~$N$, so $\lVert w_H \rVert \geq \mu$ by definition
of~$\mu$.
\end{proof}


\begin{remark}
Corollary~\ref{cor:bddStoicSubComponent} can be extended so that it
also considers the case when the closure in $[0,\infty ]^m$ of a toric
ray intersects $\ol \invtPoly \setminus \Oset$.  However, we do
not need this extension in the following section.
\end{remark}

\begin{remark}\label{rmk:MR}
M\"uller and Regensburger gave an extension of Birch's
theorem~\cite[Proposition 3.9]{MR} that is different from those
presented here (Theorems~\ref{t:birchExt1} and~\ref{thm:extBirch}).
Their result was used to generalize results about complex-balanced
systems to the setting of certain generalized mass-action ODE systems.
\end{remark}

\section{Jets}\label{s:jet}

This section introduces the technology of jets to relate the
combinatorial geometry of reaction diagrams to the dynamics.  Using
this connection, we prove that for strongly endotactic reaction
networks, the ``draining'' reactions---those that pose a threat to
persistence---are dominated by ``sustaining'' reactions
(Proposition~\ref{p:domination}).


\subsection{Jet frames, unit jets, and toric jets}\label{subsec:jet}


Recall the meaning of $\theta^w = \theta^{w_1} \cdots \theta^{w_n}$
from Notation~\ref{not:products}, where $w \in \bR^n$ and $\theta \in
\bR_{>0}$.  Recall also the geometric interpretation, from
Remarks~\ref{rmk:ray} and~\ref{rmk:polar_c},
of $w$ and $\log\theta$ together comprising polar coordinates
on~$\bR^n$.


The following notations are standard.
\begin{notation}\label{not:bigO}
The phrase ``for large $i$'' is shorthand for
``for all~$i$ greater than some fixed constant~$i_0$''.  For two
sequences $(f(i))$ and $(g(i))$ taking values in $\bR$,
\begin{enumerate}
\item%
$f(i) = O(g(i))$ if there exists $c>0$ such that $| f(i) | \leq c
|g(i)|$ for large $i$,
\item%
$f(i) = \Omega(g(i))$ if there exists $c>0$ such that $|f(i)| \geq c
|g(i)|$ for large $i$, and
\item%
$f(i) = \Theta (g(i))$ if there exist $c_1,c_2>0$ such that $c_1
|g(i)| \leq |f(i)| \leq c_2 |g(i)|$ for large $i$.
\end{enumerate}
Following standard practice, we allow the above notations to appear in
expressions and inequalities.  For instance, $f(i) \geq g(i) +
O(h(i))$ means that $f(i) \geq g(i) + k(i)$ for some $k(i) = O(h(i))$.
\end{notation}

Our next definition introduces 
\emph{jets} and \emph{toric jets}.
We first fix an orthonormal basis, a \emph{frame}.  Intuitively,
a jet consists of a sequence of points that, to first order, are
pointing in the direction of the first basis vector, with a small
perturbation in the second direction, with an even smaller
perturbation in the third direction, and so on.  A toric jet is
obtained by coordinate-wise exponentiation of a jet.

Toric jets serve as test sequences that allow us to deduce
persistence.  To prove that a certain function is Lyapunov-like, it is
not enough to analyze the function along one direction; perturbations
around that direction require exploration, as well.  It turns out that
the perturbations allowed in toric jets are general enough, while at
the same time the gradient of the function $x \log x - x$ along a
toric jet is easy to analyze; see the proof of
Theorem~\ref{thm:LyFncWorks_cpc}.


\begin{definition}\label{d:jet}
Let $n \in \ZZ_{>0}$.
\begin{enumerate}
\item%
A \emph{frame} is a list $\ol w = (w_1, \ldots, w_\ell)$ of mutually
orthogonal unit vectors in $\bR^n$.
\item%
A \emph{jet} is a sequence $\unitJet_{i \in \ZZ_{>0}}$ of vectors in
$\bR^n$ in the positive span of some frame $\ol w$ in $\bR^n$ (i.e.,
$\<w(i), w_j\> > 0$ for all $i$ and~$j$) such that for $j = 1, \ldots,
\ell-1$, the limit $\lim\limits_{i \to \infty} \frac{\< w(i),
w_j\>}{\<w(i),w_{j+1}\>}$ exists and equals $+\infty$.  A \emph{unit
jet} is a jet that consists of unit vectors.
\item\label{d:toricjet}%
A \emph{toric jet} is a sequence $\toricJet_{i \in \ZZ_{> 0}}$, where
$\unitJet$ is a unit jet in $\bR^n$ and $(\theta(i))$ is a sequence in
$\bR_{>1}$ with $\lim\limits_{i \to \infty} \thetaI = +\infty$.
\end{enumerate}
The \emph{jet frame} for the jet $\unitJet$ or toric jet $\toricJet$
is $\ol w$, and these jets are \emph{framed by}~$\ol w$.
\end{definition}


\begin{remark}\label{rmk:testSeq}
Toric jets have a geometric interpretation: the sequence
\begin{align}\label{eq:seq_toric}
  \big(\log \thetaN\big) \wN
  =
  (\log \thetaN) \big(\bCoef{1}{n} w_1 + \cdots + \bCoef{\ell}{n} w_\ell\big)
\end{align}
is to first approximation a sequence going to infinity in direction
$w_1$ in $\bR^n$, which is viewed as ``energy space''; the
second-order correction is in direction~$w_2$, and so on.
Exponentiating the sequence~\eqref{eq:seq_toric} yields the toric jet
$\toricJet$ in ``concentration space''~$\Rplus^n$.

In the context of a reaction network, if the image of a toric jet lies
in a given invariant polyhedron~$\invtPoly$, then the toric jet
necessarily approaches the boundary of $\invtPoly$ or is unbounded.
More precisely, $\toricJet$ approaches the boundary of the closure
$\ol\invtPoly$ of $\invtPoly$ in the compactification $[0,
\infty]^\SS$.  Our approach to proving persistence is to demonstrate
that no toric jet is a sequence of points along a trajectory (see the
proof of Theorem~\ref{thm:LyFncWorks_cpc}).  This strategy is a
distillation of the argument employed by Anderson in
\cite{Anderson11}.
\end{remark}

\begin{remark}
The concepts of jet frame and unit jet were introduced by Miller and
Pak in~\cite[Definition 4.1]{MP} for the purpose of describing the
interaction of an expanding wavefront on the boundary of a convex
polytope infinitesimally after the wavefront hits a new face.  Our
definitions are related but not identical to theirs.
\end{remark}

\begin{example}\label{ex:toric ray}
Every sequence of points $\theta(i)^w$ along a toric ray with
$\theta(i) \to \infty$ is a special case of a toric jet.  Here the jet
frame is the singleton list containing only the direction $w$ of the
toric ray.
\end{example}

We now give two lemmas concerning unit jets.


\begin{lemma}\label{l:w1}
If $\unitJet$ is a unit jet with frame $\ol w = (w_1,\ldots,w_\ell)$,
then $\wI \to w_1$ as $i \to \infty$.
\end{lemma}
\begin{proof}
Let $w(i) = \sum_{j=1}^\ell \beta_j(i) w_j$, where $\bCoef{j}{i} > 0$.
By definition, $\frac{\beta_j(i)}{\beta_{j+1}(i)} \to \infty$ for $j =
1,\ldots,\ell-1$.  As every $w(i)$ and $w_j$ is a unit vector, it
follows that $\beta_1(i) \to 1$, and all other $\beta_j(i)\to 0$ (for
$j = 2,\ldots,\ell$).
\end{proof}

\begin{lemma}[Unit jets are abundant]\label{l:jet}
Let $n\in \ZZ_{>0}$.  Every infinite sequence of unit vectors in
$\bR^n$ has an infinite subsequence that is a unit jet.
\end{lemma}
\begin{proof}
Let $\unitJet$ be an infinite sequence of unit vectors in $\bR^n$.
Throughout this proof, subsequences of $\unitJet$ are denoted again by
$\unitJet$ for ease of notation.  We proceed by induction on $n$.

For $n=1$, if the sequence $\unitJet$ takes the value $+1$ infinitely
often, then choose $w_1$ to be $+1$.  Otherwise $w(i)$ takes the value
$-1$ infinitely often, so choose $w_1$ to be $-1$.  The required unit
jet is the constant subsequence $(w(i) = w_1)$ with frame $\ol w =
(w_1)$.

Now assume that $n \geq 2$.
As the sequence $w(i)$ lies in the unit sphere, which is compact, the
sequence must have an accumulation point, which we denote by $w_1$.
Restricting to a subsequence if necessary, assume that $\lim\limits_{i
\to \infty} w(i) = w_1$.

Let $\bCoef{1}{i} = \<w(i), w_1\>$.
Then $\lim\limits_{i \to \infty}\bCoef{1}{i} = \< w_1, w_1 \> =1$.
Take a subsequence with $\bCoef{1}{i} >0$ for all $i$.  Consider the
sequence $(w'(i) = w(i) - \bCoef{1}{i} w_1 )$.

\noindent
\textbf{Case~1.}  For large $i$, the sequence $(w'(i))$ is the zero
vector.  Then for large $i$, $w(i)=\bCoef{1}{i}w_1$, so this
subsequence is a unit jet with frame $\ol w = (w_1)$.

\noindent
\textbf{Case~2.}  The sequence $(w'(i))$ is nonzero infinitely often.
Then restrict to a subsequence that is always nonzero.  By induction,
the sequence of unit vectors
$w'(i)/||w'(i)||$ in $w_1^\perp \cong \bR^{n -1}$ has a subsequence
with frame $(w_2,\ldots,w_\ell) \subseteq w_1^\perp$ and corresponding
positive coefficients~$\widetilde\beta_j(i)$ for $j = 2,\ldots,\ell$
such that
$w'(i)/||w'(i)|| = \sum_{j=2}^\ell\widetilde\beta_j(i)w_j$ is a unit
jet.

For $j = 2,\ldots,\ell$, let $\bCoef{j}{i}=\widetilde\beta_j(i) \cdot
||w'(i)||$.  We claim that the corresponding subsequence defined by
$w(i) = \bCoef{1}{i} w_1 + \sum_{j=2}^\ell \bCoef{j}{i}w_j$ is a unit
jet framed by $(w_1, \ldots, w_\ell)$.  We need only check that the
first coefficient dominates the second.  To see that it does, note
that $\bCoef{2}{i} = \widetilde\beta_2(i) \cdot \lVert w'(i) \rVert$
converges to~$0$, because $\widetilde\beta_2(i) \to 1$ by definition
of jet frame and $||w'(i)||\to 0 $ by definition of $w_1$ and
$\beta_1$.  Hence $\bCoef{1}{i} / \bCoef{2}{i} \to \infty$ because
$\bCoef{1}{i} \to 1$.
\end{proof}

\subsection{Geometry of jets}\label{subsec:geo_jets}

Theorem~\ref{l:vj} in this subsection conveys an important geometric
idea behind jets.  Later we use this result to prove a useful
characterization of endotactic networks (Lemma~\ref{l:endjets}).  We
begin with some preliminaries.


\begin{definition}\label{d:super}
Fix a positive integer $n \in \Zplus$, a finite set $Q \subseteq
\bR^n$, and a list $\ol w = (w_1,\ldots,w_\ell)$ of vectors
in~$\bR^n$.  Set $\TT_0(Q,\ol w) = Q$, and inductively for $j =
1,\ldots,\ell$ let $\TT_j(Q,\ol w)$ denote the set of
$\leq_{w_j}$-maximal elements of $\TT_{j-1}(Q,\ol w)$.
\end{definition}

\begin{notation}
When the set $Q$ and the list $\ol w$ of vectors is clear from
context, we write $\TT_j$ to denote $\TT_j(Q,\ol w)$.
\end{notation}


\begin{lemma}\label{l:mmlrat}
Fix a positive integer $n \in \Zplus$, a frame $\ol w =
(w_1,\ldots,w_\ell)$ in $\bR^n$, a unit jet $(w(i) = \sum_{j=1}^\ell
\beta_j(i) w_j)$ framed by~$\ol w$, and a finite set $Q \subseteq
\bR^n$.  Let $\lambda\in \{1,\ldots,\ell\}$.
\begin{enumerate}
\item\label{l:mmlrat.1}%
If $x\in \TT_\lambda$ and $y \notin \TT_\lambda$, then there
exists $k \in \{1,\ldots,\lambda\}$ such that $\<w_k, x-y\> >0$ and
$\<w_j, x-y\> = 0$ for all $j = 1,\ldots,k-1$.  Consequently, $\<w(i),
x-y\> > 0$ for large~$i$, and $\<w(i), x-y\> = \Omega(\beta_k(i))$.
\item\label{l:mmlrat.2}%
If $\lambda > 1$ and $x \in \TT_{\lambda-1}$ and $y \in Q$, then for
large~$i$,
\begin{align}\label{eq:start_at_lam}
  \<w(i),x-y\> \geq \sum_{j=\lambda}^\ell \beta_j(i) \<w_j, x-y\>.
\end{align}
\end{enumerate}
\end{lemma}
\begin{proof}
1.  Suppose $x \in \TT_\lambda$ and $y \notin \TT_\lambda$.  Let $k
\in \{1,\ldots,\lambda\}$ be the smallest positive integer such that
$y\notin \TT_k$.  It follows that $\< w_k,x-y\> > 0$ because $y \in
\TT_{k-1} \setminus \TT_k$ and $x \in \TT_k$.  Further, $\<w_j,x-y\> =
0$ for all $j < k$, since $x,y \in \TT_j$ for $j < k$.  Thus
$$%
  \<w(i),x-y\>
  \ =\
  \sum_{j=1}^{\ell} \beta_j(i) \<w_j,x-y\>
  \ =\
  \beta_k(i) \<w_k,x-y\> + O(\beta_{k+1}(i)),
$$
where $\beta_{\ell+1}(i)$ is understood to be~$0$.  Thus,
$\<w(i),x-y\> > 0$ for large $i$ because $\<w_k,x-y\> > 0$ and, by
definition of unit jet, $\beta_k(i) >0$ and $\beta_{k+1}(i)/\beta_k(i)
\to 0$ as $i \to \infty$.  Additionally, it follows that $\<w(i),x-y\>
= \Omega(\beta_k(i))$.

2.  Suppose $x \in \TT_{\lambda-1}$.  Set
$$%
  \phi(i) = \sum_{j=1}^{\lambda-1} \beta_j(i) \<w_j,x-y\>,
$$
and rewrite inequality~\eqref{eq:start_at_lam} as
$\phi(i) \geq 0$.
If $y\in \TT_{\lambda-1}$, then $\<w_j,x-y\>=0$ for $j =
1,\ldots,\lambda - 1$, so in fact $\phi(i) = 0$.
If $y \notin \TT_{\lambda-1}$, then consider the sequence defined by
$$%
  v(i) = \frac{\sum_{j=1}^{\lambda-1} \beta_j(i) w_j}{||
  \sum_{j=1}^{\lambda-1} \beta_j(i) w_j||},
$$
which is a unit jet framed by $(w_1,\ldots,w_{\lambda-1})$ by
construction, using the fact that $(w(i))$ is a unit jet.  Then
\begin{align}\label{eq:start_at_lam_3}
  \phi(i)
  =
\Big\|\sum_{j=1}^{\lambda-1} \beta_j(i) w_j \Big\| \<v(i), x-y\>.
\end{align}
Of the terms in~\eqref{eq:start_at_lam_3}, $\sum_{j=1}^{\lambda-1}
\beta_j(i) w_j$ approaches $w_1$ (by definition of unit jet), so its
norm $||\sum_{j=1}^{\lambda-1} \beta_j(i) w_j||$ approaches~$1$, and
$\<v(i), x-y\> > 0$ for large $i$ by part~1 of this lemma.  So
$\phi(i) > 0$ for large $i$, and inequality~\eqref{eq:start_at_lam}
holds.
\end{proof}

The next result and Corollary~\ref{cor:polytope} may be seen as
generalizations of Proposition~2.3(iii) in Ziegler's \emph{Lectures on
Polytopes}~\cite{Ziegler}; with more careful accounting, we obtain our
required stronger result.  An alternative, direct proof of
Corollary~\ref{cor:polytope} follows the same lines as the proof in
Ziegler's book, without appealing to jets.


\begin{theorem}[Fundamental theorem of jets]\label{l:vj}
Fix a positive integer $n\in\Zplus$, a frame $\ol w =
(w_1,\ldots,w_\ell)$ in~$\bR^n$, a unit jet $\big(v(i) =
\sum_{j=1}^\ell \beta_j(i) w_j\big)$ framed by~$\ol w$, and a finite
set $Q \subseteq \bR^n$.  Then for large $i$, the
$\leq_{v(i)}$-maximal subset of $Q$ equals $\TT_\ell$.
\end{theorem}
\begin{proof}
The proof is by induction on~$\ell$.  More precisely, for
$j=1,\ldots,\ell$, let $(v_j(i))$ be~the jet
$$%
  v_j(i) = \beta_1(i) w_1 + \dots + \beta_j(i) w_j.
$$
We
prove that for all $j = 1,\ldots,\ell$, the $\leq_{v_j(i)}$-maximal
subset of~$Q$ equals $\TT_j$ for large $i$.  The base case of the
induction---that $\TT_1$ is the $\leq_{v_1(i)}$-maximal subset of~$Q$
for large~$i$---is true by definition of~$\TT_1$ because $v_1(i) \to
w_1$.

Fix $j < \ell$ and assume that the $\leq_{v_j(i)}$-maximal subset of
$Q$ equals $\TT_j$ for large $i$.  From now on, consider only such
sufficiently large $i$.  For $x\in\TT_j$, let $c{}(i) = \<v_j(i),x\>$
be the value of the inner product of $v_j(i)$ with every element in
$\TT_j$; informally, write $c{}(i) = \<v_j(i),\TT_j\>$.  Similarly,
let $c' \in \bR$ denote the value of the inner product of $w_{j+1}$
with every element in $\TT_{j+1}$; informally, write $c' = \<w_{j+1},
\TT_{j+1}\>$.  Next, define $\delta(i)$ to be the reciprocal of
left-hand side of the inequality
\begin{align}\label{eq:def_delta_i}
  \frac{1}{\delta(i)} =
  \max_{y\in Q\setminus\TT_j}
    \frac{\max\{0, \<w_{j+1}, y\> - c'\}}{c{}(i) - \<v_j(i),y\>}
  \leq
  \frac{\max_{y\in Q\setminus\TT_j} \{0, \<w_{j+1}, y\>-
  c'\} }{\min_{y\in Q\setminus\TT_j} (c{}(i) - \<v_j(i), y\>)} .
\end{align}
If $1/\delta(i) = 0$, then declare $\delta(i) = +\infty$.  Note that
$\delta(i) > 0$ for large~$i$, because the denominators
in~\eqref{eq:def_delta_i} are strictly positive by the inductive
hypothesis, and the numerators are nonnegative by construction.  The
numerator on the right-hand side of~\eqref{eq:def_delta_i} has no
dependence on $i$, and hence is $O(1)$.  Also, using the identity
$\<v_j(i), x\> = c{}(i)$ for $x\in\TT_j$, the denominator on the
right-hand side of~\eqref{eq:def_delta_i} satisfies
\begin{align*}
  \min_{y \in Q\setminus\TT_j} \big(c{}(i) - \<v_j(i),y\>\big)
  =
  \min_{\substack{y \in Q\setminus\TT_j\\x \in \TT_j\\[-.25ex]}} \<v_j(i),x-y\>
  =
  \Omega\big(\beta_j(i)\big),
\end{align*}
where the second equality follows from
Lemma~\ref{l:mmlrat}.\ref{l:mmlrat.1}.  Thus the
inequality~\eqref{eq:def_delta_i} yields
$$%
\delta(i)
  \geq
  \frac{\min_{y\in Q\setminus\TT_j} \big(c{}(i) - \<v_j(i),y\>\big)}
  {\max_{y\in Q\setminus\TT_j}\{0,\<w_{j+1}, y\>-c'\}}
  =
  \frac{\Omega\big(\beta_j(i)\big)}{O(1)}
  =
  \Omega\big(\beta_j(i)\big).
$$
Hence, by the definition of jet, $\delta(i)>\beta_{j+1}(i)$ for large
$i$.

We now complete the proof by proving that for large~$i$, the following
inequality holds for all $y \in Q$ and achieves equality exactly on
the set~$\TT_{j+1}$:
\begin{align}\label{eq:inequality_j+1}
\<v_{j+1}(i), y\>
   = \<v_j(i), y\> + \beta_{j+1}(i) \<w_{j+1}, y\>
\leq c{}(i) + \beta_{j+1}(i) c'.
\end{align}
If $y \in \TT_j$, then $\<v_j(i),y\> = c{}(i)$
by definition and $\<w_{j+1}, y\> \leq c'$ by construction, with
equality holding precisely when $y \in \TT_{j+1}$; hence the desired
(in)equality in~\eqref{eq:inequality_j+1} holds.
Now assume that $y \in Q \setminus \TT_j$.  There are two subcases
based on the sign of $\<w_{j+1}, y\> - c'$.  First, if $\<w_{j+1}, y\>
- c' <0$, then the inequality in~\eqref{eq:inequality_j+1} is strict,
noting that $\<v_j(i), y\> \leq c{}(i)$ because $y \in Q$.
In the remaining case, when $y \in Q \setminus \TT_j$ and $\<w_{j+1},
y\> - c' \geq 0$, the construction of $\delta(i)$
in~\eqref{eq:def_delta_i} yields the first of the inequalities
\begin{align}\label{eq:ineq_case2b}
  c{}(i) - \<v_j(i), y\>
  \geq
  \delta(i) \big(\<w_{j+1}, y\> -c'\big)
  >
  \beta_{j+1}(i) \big(\<w_{j+1}, y\> -c'\big),
\end{align}
while the second inequality is due to the inequality $\delta(i) >
\beta_{j+1}(i)$ proven earlier.  Rearranging~\eqref{eq:ineq_case2b}
yields the desired strict inequality in~\eqref{eq:inequality_j+1}.
\end{proof}

\begin{corollary}\label{cor:polytope}
Fix a positive integer $n\in\Zplus$, a frame $\ol w =
(w_1,\ldots,w_\ell)$ in~$\bR^n$, and a finite set $Q \subseteq \bR^n$.
There exist $\beta_1,\ldots,\beta_\ell \in \Rplus$ and a positive
linear combination $\widetilde w = \beta_1 w_1 + \cdots + \beta_\ell
w_\ell$ such that the set of $\leq_{\widetilde w}$-maximal elements
of~$Q$ equals~$\TT_\ell$.  Additionally, if
$\gamma_1,\ldots,\gamma_\ell \in \Rplus$,
then $\beta_1,\ldots,\beta_\ell$ can be chosen so that $\beta_j <
\gamma_j \beta_{j-1}$ for all $j = 2,\dots,\ell$.
\end{corollary}
\begin{proof}
Use Theorem~\ref{l:vj}: first consider a jet $(w(i))$ in the frame
$\ol w$, and then define $\widetilde{w}$ as $w(i_0)$ for a
sufficiently large value of~$i_0$, and use the definition of jet.
\end{proof}


\begin{remark}
The hypotheses of Corollary~\ref{cor:polytope} can be weakened.  The
corollary remains true for any set $Q$ whose convex hull is a polytope
and for any list of vectors $(w_1,\ldots,w_\ell)$ that need not be
orthonormal.  Indeed, this is the setting of the analogous result in
Ziegler's book~\cite[Proposition~2.3(iii)]{Ziegler}.  We do not need
this generality here.
\end{remark}

\begin{example}\label{ex:hexagon}
We illustrate Corollary~\ref{cor:polytope} by considering the set $Q$
consisting of the six vertices of a hexagon:
$$%
\begin{xy}<15mm,0cm>:
  (-.25,0)="labelQ" *{Q} ;
  (0, \sinpioverthree) ="123" *{\bullet};
  (.75 , \halfSinPiOverThree) ="132" *{\bullet} *+!DL{Q_2};
  (.75, -\halfSinPiOverThree) ="312" *{\bullet};
  (0, -\sinpioverthree)  ="321" *{\bullet};
  (-.75,-\halfSinPiOverThree) ="231" *{\bullet};
  (-.75, \halfSinPiOverThree) ="213" *{\bullet};
  (-3,-.5)="corner" *+!UL{}  *{}; 
  (-2,-.5) 	="cornerRight" *+!R{ }  *{}; 
  (-3,.5) 	="cornerUp" *+!R{ }  *{}; 
  (-2,.1) 	="cornerMiddle" *+!L{}  *{}; 
  (-2,1.23205) = "cornerDotted" *+!R{ }  *{};
     "123";"132" **@{-};
     "132";"312" **@{-}; 
     "312";"321" **@{-};
     "321";"231" **@{-};
     "231";"213" **@{-};
     "213";"123" **@{-};
  (0.75, 0) *+!L{Q_1};
  {\ar@{->}_{w_1}"corner";"cornerRight"  *{}  }; 	
  {\ar@{->}^{w_2}"corner";"cornerUp"  *{}  }; 	
  {\ar@{->}_{w_1+ \varepsilon w_2}"corner";"cornerMiddle"  *{}  }; 	
  {\ar@{.>}_{w_1 + \delta w_2}"corner";"cornerDotted" *{}  }; 	
  {\ar@{.>}^{}(-.25,-.5);(.75,1.23205)  *{}  };
\end{xy}
$$
The set $Q_1$ consists of the two vertices on the edge that is the
$\leq_{w_1}$-maximal face of the hexagon, and the vertex $Q_2$ is the
$\leq_{w_2}$-maximal subset of $Q_1$.  For $0<\varepsilon<\delta
=\sqrt{3}$, the vertex $Q_2$ is the $\leq_{w_1 + \varepsilon
w_2}$-maximal subset of $Q$.  In other words, $w_1 + \varepsilon w_2$
defines $Q_2$ as long as it lies between $w_1$ and the dotted arrow
depicted twice in the figure above.
\end{example}


\subsection{Jets and reactions}\label{sub:jets+reactions}

One of our primary
goals, carried out in Section~\ref{s:prelya}, is to prove that given
a strongly endotactic network, there is a compact set outside of which
the function $g(x) = \sum_{i\in\SS} x_i \log x_i - x_i$ decreases
along trajectories; that is, its derivative along trajectories is
negative (Theorem~\ref{thm:LyFncWorks_cpc}).  Intuitively, this
derivative is a sum of contributions (``pulls'';
Definition~\ref{d:pull}) from each reaction (see the proof of
Lemma~\ref{l:preL}).  A ``draining'' reaction acts to hinder our
efforts (Definition~\ref{d:level}) with a positive pull.  The pull
of a ``sustaining'' reaction is negative.
Proposition~\ref{p:domination} shows that, after fixing a toric jet,
each draining reaction is dominated by some sustaining reaction,
so negativity prevails overall and $g(x)$ indeed decreases.

\begin{definition}\label{d:level}
For a reaction network $\GG$, let $\ol w = (w_1,\ldots,w_\ell)$ be a frame
in~$\bR^\SS$.
\begin{enumerate}
\item%
A reaction $y \to y' \in \RR$ is \emph{$\ol w$-essential} if $\<w_j,
y'-y\> \neq 0$ for some $j = 1,\ldots,\ell$.
\item%
The \emph{level} of a $\ol w$-essential reaction $y \to y'$ is the
least $j$ such that $\<w_j, y'-y\> \neq 0$.
\item%
A $\ol w$-essential reaction $y \to y'$ with level $\lambda$ is
\begin{enumerate}
\item%
$\ol w$-\emph{sustaining} if $\<w_\lambda, y'-y\> < 0$.
\item%
$\ol w$-\emph{draining} if $\<w_\lambda, y'-y\> > 0$.
\end{enumerate}
\end{enumerate}
\end{definition}


\begin{remark}\label{rmk:reless}
For a reaction network $\GG$ and a nonzero vector $w$ in $\bR^\SS$, a
reaction $y \to y' \in \RR$ is $(w)$-essential
(Definition~\ref{d:level}.1) if and only if $y \to y'$ is
$w$-essential (Definition~\ref{d:partialOrder}.3).
More generally, for a frame $\ol w = (w_1,\ldots,w_\ell)$
in~$\bR^\SS$, a reaction $y \to y'$ is $\ol w$-essential if and only
if $y \to y'$ is $w_j$-essential for some $j = 1,\ldots,\ell$.
\end{remark}


\begin{example}\label{ex:goodBad}
For network $G$ of Figure~\ref{fig:end} in Section~\ref{s:CRNT}, let
$\ol w = (w_1,w_2)$ be a frame with $w_1=w$, where $w$ is depicted in
Figure~\ref{fig:end}.  Then reaction $y_2 \to y_2'$ is $\ol
w$-draining, and $y_3 \to y_3'$ is $\ol w$-sustaining.  Whether $y_1
\to y_1'$ is $\ol w$-\good or $\ol w$-\bad depends on whether $\<w_2,
y_1'-y_1\>$ is negative or positive.
\end{example}


\begin{definition}\label{d:pull}
Let $G=\GG$ be a reaction network.
\begin{enumerate}
\item%
The \emph{pull} of a reaction $y \to y'\in\RR$ along a toric jet
$\toricJet$ in $\bR^\SS$ is the function
\begin{align}\label{eq:pull}
  \pull\limits_{y \to y'}(i) = \<w(i), y'-y\>\theta(i)^{\<w(i), y\>}.
\end{align}
\item%
A unit jet $(w(i))$ in $\bR^\SS$ is \emph{adapted to $G$} if for
every reaction $y \to y'\in\RR$, either
\begin{enumerate}
\item%
$\<w(i), y'-y\>=0$ for all~$i$; or
\item%
$\<w(i), y'-y\><0$ for all~$i$,
in which case $y \to y'$ is \emph{sustaining along $(w(i))$}; or
\item%
 $\<w(i), y'-y\>>0$ for all~$i$,
in which case $y \to y'$ is \emph{draining along $(w(i))$}.
\end{enumerate}
A reaction $y \to y'$ is \emph{essential along $(w(i))$} if it is
sustaining or draining along $(w(i))$.
\item%
A toric jet $\toricJet$ framed by $(w_1, \ldots, w_\ell)$ in $\bR^\SS$
is \emph{adapted to $G$} if its unit jet $(w(i))$ is adapted to $G$
and, for all $j = 1,\ldots,\ell$, the limit $\lim\limits_{i\to\infty}
\theta(i)^{\<w(i), w_j\>}$ exists in~$[1,\infty]$.
\item%
A reaction $r\in\RR$ \emph{dominates} a reaction $r'\in\RR$
\emph{along} a toric jet $\toricJet$ in $\bR^\SS$ if the ratio of
their pulls along $\toricJet$ tends to infinity in absolute value:
$\lim\limits_{i \to \infty}
\Big|\frac{\pull\limits_r(i)}{\pull\limits_{r'}(i)}\Big|$ exists and
equals $+ \infty$.
\end{enumerate}
\end{definition}
\noindent
When the toric jet is clear from context, we simply say that reaction
$r$ \emph{dominates} reaction~$r'$.


\begin{lemma}\label{l:adapted}
Fix a reaction network $\GG$.  Let $\wI$ be a unit jet in~$\bR^\SS$ and
let $\toricJet$ be a toric jet in $\bR^\SS$.
\begin{enumerate}
\item%
There exists an infinite subsequence of $\wI$ that is adapted to~$\GG$.
\item%
There exists an infinite subsequence of $\toricJet$ that is adapted
to~$\GG$.
\end{enumerate}
\end{lemma}
\begin{proof}
Repeatedly take subsequences, first so that for each reaction $y \to
y' \in \RR$, the sign of $\<w(i), y'-y\>$ is constant for all~$i$, and
then so that for all~$j = 1,\ldots,\ell$, an accumulation point of
$\thetaN^{\<w(i), w_j\>}$ becomes a limit point.  We can do this
because $\thetaN > 1$ and $\<w(i), w_j\> >0$ for all~$i$, whence
$\thetaN^{\<w(i), w_j\>}$ lies in the compact set $[1,\infty]$.
\end{proof}


It is equivalent for a reaction to be essential along a unit jet or
with respect to its jet frame, if the unit jet is adapted to the
relevant network.  Here is a more precise statement.

\begin{proposition}\label{p:relrn}
Let $G = \GG$ be a reaction network, and let $\unitJet$ be a unit jet
adapted to~$G$ and framed by $\ol w = (w_1,\ldots,w_\ell)$.  Let
$\lambda\in\{1,\ldots,\ell\}$.
\begin{enumerate}
\item\label{relrn1}%
A reaction is $\ol w$-sustaining (respectively, $\ol w$-draining) if
and only if it is sustaining (respectively, draining) along $(w(i))$.
\item%
A $\ol w$-essential reaction $y \to y'\in\RR$ has level $\lambda$ if
and only if $|\<w(i), y' - y\>| = \Theta\big(\<w(i),
w_\lambda\>\big)$.
\end{enumerate}
\end{proposition}
\begin{proof}
Let $w(i) = \sum_{j=1}^\ell \beta_j(i) w_j$.  Consider a reaction $y \to
y' \in\RR$.

If $y \to y'$ is not $\ol w$-essential then $\<w_j,y'-y\> = 0$ for
all~$j$, and in that case $\<w(i),y'-y\> = 0$ for all $i$, so $y \to
y'$ is not essential along $(w(i))$.

Next assume $y \to y'$ is $\ol w$-essential with level $\lambda$, so
$\<w_\lambda,y'-y\> = 0$ for all $j < \lambda$ whereas
$\<w_\lambda,y'-y\> \neq 0$, where the inequality is ``$<$'' in the
sustaining case and ``$>$'' in the draining case.  Thus $\<w(i),
y'-y\> = \beta_\lambda(i) \<w_\lambda, y'-y\> +
O(\beta_{\lambda+1}(i))$, so $\<w(i), y'-y\>$ is nonzero for large $i$
by the definition of unit jet---again, with $< 0$ in the sustaining
case and $> 0$ in the draining case.  Hence $y \to y'$ is essential
along $\unitJet$ if it is $\ol w$-sustaining and draining along
$\unitJet$ if it is $\ol w$-draining, and in either case $|\<w(i), y'
- y\>| = \Theta(\beta_\lambda(i))= \Theta(\<w(i), w_\lambda\>)$.  That
takes care of the ``$\implies$'' direction of part~2.

For the ``$\Leftarrow$'' direction of part 2, suppose that the
reaction $y \to y'$ is $\ol w$-essential and that $|\<w(i), y' - y\>|
=\Theta(\beta_\lambda(i))$.  It follows that $\<w_j, y' - y\> = 0$ for
all $j<\lambda$, and $\<w_\lambda, y' - y\> \neq 0$, so $y \to y'$ has
level~$\lambda$.
\end{proof}


\subsection{Jets and endotactic networks}\label{sub:jets+networks}

The next two results interpret endotactic and strongly endotactic
networks in terms of~jets.

\begin{notation}\label{not:super}
For a reaction network $\GG$ and jet frame $\ol w = \{w_1,\ldots,
w_\ell\}$, denote by $\TT_1$ the $\leq_{w_1}$-maximal subset of
$\source{\RR}$, and for $j\in\{2,\ldots,\ell\}$, write $\TT_j$ for the
$\leq_{w_j}$-maximal subset of $\TT_{j-1}$.
\end{notation}

\noindent
In other words, the sets $\TT_j$ coincide with those in
Section~\ref{subsec:geo_jets}, where $Q = \source\RR$.


\begin{lemma}\label{l:endjets}
Let $\GG$ be a reaction network.
\begin{enumerate}
\item%
$\GG$ is endotactic if and only if for every singleton frame $(w_1)$
in $\bR^\SS$ and for every $(w_1)$-draining reaction $y \to y'$ in
$\RR$, there exists a $(w_1)$-sustaining reaction $x \to x'$ such that
$\<w_1, x - y\> > 0$.
\item%
If $\GG$ is endotactic, then for every frame $\ol w$ in $\bR^\SS$ and
every $\ol w$-draining reaction $y \to y'$ in $\RR$ with level
$\lambda$, the reactant $y$ lies outside of\/ $\TT_\lambda$.
\end{enumerate}
\end{lemma}
\begin{proof}
1.\
The network $\GG$ fails to be endotactic if and only if there is a
unit vector~$w_1$ in~$\bR^\SS$ and a reaction $y \to y' \in \RR$ with
$y \in \op{supp}_{w_1}\GG$ such that $\<w_1, y'-y\> > 0$ (recall
Definition~\ref{d:endotactic}).  Equivalently, $\GG$ fails to be
endotactic precisely when there is a unit vector $w_1$ in $\bR^\SS$
and a $(w_1)$-draining reaction $y \to y' \in \RR$ such that every
$(w_1)$-sustaining reaction $x \to x'$ satisfies $y \leq_w x$.
Now use that $y \leq_w x \iff \<w_1, x-y\> \leq 0$.

2.\ Suppose $\GG$ is endotactic.  Consider a frame $\ol w =
(w_1,\ldots,w_\ell)$ in $\bR^\SS$.  Suppose $y \to y'$ is $\ol
w$-draining and has level~$\lambda$.
By Corollary~\ref{cor:polytope}, $\TT_\lambda$ is the
$\leq_{w^*}$-maximal subset of $\TT_0 = \source\RR$, where $w^* = w_1
+ \varepsilon_1 w_2 + \cdots + \varepsilon_{\lambda-1} w_\lambda$ for
some positive numbers $\varepsilon_1,\ldots,\varepsilon_{\lambda-1}$.
Corollary~\ref{cor:polytope} allows us to choose all of
$\varepsilon_1,\ldots,\varepsilon_{\lambda-1}$ to be arbitrarily small, so
we can ensure that $\<w^*,y'- y\>$ has the same sign as
$\<w_1,y'-y\>$, which is positive.  Then, from the definition of
$w^*$-endotactic, we conclude that $y \notin \TT_\lambda$.
\end{proof}

\begin{example}\label{ex:converse_endo_lem_false}
The converse of Lemma~\ref{l:endjets}.2 is false.  For the
1-dimensional network
\begin{align}\label{eq:net_converse}
\begin{xy}<14mm,0cm>:
  (1,0) 	="y1proj" *{\bullet}*+!R{{y_1} = {y_1}' ~ }  *{}; 
  (1,.02) 	="y1proj_2" *+!U{\circlearrowleft }  *{}; 
  (2.25,0) 	="y2proj" *+!R{{y_2}' }  *{}; 
  (3,0) 	="y2'proj" *{\bullet}*+!L{{y_2}}  *{}; 
  (5.5,0) 	="y4proj" *{\bullet}*+!L{{y_3}}  *{}; 
  (4.75,0) 	="y4'proj" *+!R{{y_3}'}  *{}; 
  {\ar "y2'proj";"y2proj"*{}  }; 		
  {\ar "y4proj";"y4'proj"*{}  };
  {\ar@{<--}^{w}(-1.75,0);(-.75,0)  *{}  }; 		
  (-.25,-.4) 	="Y1" *+!R{ }  *{}; 
  (6.4,-.4) 	="Y2" *+!R{ }  *{}; 
  (6.4, .3) 	="Y3" *+!UR{  }  *{}; 
  (-.25,.3) 	="Y4" *+!R{ }  *{}; 
  "Y1";"Y2" **\dir{.};
  "Y2";"Y3" **\dir{.};
  "Y3";"Y4" **\dir{.};
  "Y4";"Y1" **\dir{.};
\end{xy}
\end{align}
the essential reactions (if any) corresponding to both (singleton)
frames $(w)$ and $(-w)$ are sustaining.  More specifically, there are
no $(w)$-essential reactions (although $y_1$ is the $\leq_{w}$-maximal
reactant, the reaction $y_1 \to y_1'$ is not essential), and the
unique $(-w)$-essential reaction $y_3 \to y_3'$ is $(-w)$-sustaining.
However, network~\eqref{eq:net_converse} is not endotactic: $y_2 \to
y_2'$ is the leftmost reactant among all of the nontrivial reactions,
but it points to the left.
\end{example}


\begin{proposition}\label{p:stendjets}
For an endotactic reaction network $\GG$ with stoichiometric
subspace~$H$, the following are equivalent.
\begin{enumerate}
\item%
$\GG$ is strongly endotactic.
\item%
For every singleton jet frame $(w_1)$ in $\bR^\SS$ with $w_1 \notin
\HPerp$, there exists a $(w_1)$-sustaining reaction $x \to x'$ such
that $x \in \TT_1$.
\item%
For every frame $\ol w = (w_1,\ldots,w_\ell)$ in $\bR^\SS$ with
$w_1\notin H^\perp$, there exists a $\ol w$-sustaining reaction $x \to
x'$ with $x\in \TT_\ell$.
\end{enumerate}
\end{proposition}
\begin{proof}
The equivalence of items~1 and~2 is straightforward from the
definition of strongly endotactic, and item~2 is a special case of
item~3.  We therefore assume item~1, with the goal of deducing item~3.
Let $\ol w =(w_1,\ldots,w_\ell)$ be a frame in $\bR^\SS$ with $w_1
\notin \HPerp$.  Let $(w(i))$ be a unit jet framed by $\ol w$.  Use
Lemma~\ref{l:adapted}.1 to pick a subsequence of $(w(i))$ adapted to
$\GG$.  Using the fundamental theorem of jets (Theorem~\ref{l:vj}),
take $i_0$ large enough so that the $\leq_{w(i_0)}$-maximal subset of
$\source \RR$ equals $\TT_\ell$.  Taking $i_0$ even larger, if
necessary, assume that $w(i_0) \notin H^\perp$, which is possible
because $w(i) \to w_1\notin H^\perp$.  $\GG$ is strongly
$w(i_0)$-endotactic, so there exists a $w(i_0)$-sustaining reaction $x
\to x'$ with $x\in\TT_\ell$.  Since $(w(i))$ is a unit jet adapted to
$\GG$, by definition $x \to x'$ is sustaining along the unit jet
$(w(i))$.  Proposition~\ref{p:relrn}.\ref{relrn1} implies that $x \to
x'$ is $\ol w$-sustaining.
\end{proof}


\begin{remark}
Proposition~\ref{p:stendjets} can be proven directly, using only
Corollary~\ref{cor:polytope}.
\end{remark}

Here is the main result of this section.


\begin{proposition}\label{p:domination}
Fix a strongly endotactic reaction network $G = \GG$ with
stoichiometric subspace~$H$.  Let $\toricJet$ be a toric jet adapted
to~$G$ framed by $\ol w = (w_1,\ldots,w_\ell)$ in $\bR^\SS$.  If $w_1
\notin \HPerp$ then every draining reaction along $(w(i))$ is
dominated by a sustaining reaction along~$(w(i))$.
\end{proposition}
\begin{proof}
Suppose $w_1\notin\HPerp$.  Let $w(i) = \sum_{j=1}^\ell \beta_j(i)
w_j$.  Let $y \to y' \in \RR$ be a draining reaction along $\unitJet$.
We must find a sustaining reaction $x \to x'$ that dominates $y \to
y'$; that is
\begin{align}\label{eq:pullratio}
\bigg|\frac{\pull_{x \to x'}(i)}{\pull_{y \to y'}(i)}\bigg|
=
\bigg|\frac{\<w(i),x'-x\>}{\<w(i),y'-y\>}\bigg|\theta(i)^{\<w(i), x-y\>}\to\infty.
\end{align}
For $k = 1,\ldots,\ell$, define a sequence $(v_k(i))$ of unit vectors
by
\begin{align}\label{eq:seq_vk}
v_k(i) = \frac{\sum_{j=1}^k \beta_k(i) w_j}{\big\lVert \sum_{j=1}^k
\beta_k(i) w_j \big\rVert}.
\end{align}
For $k = 1,\ldots,\ell$, the sequence $(v_k(i))$ is a unit jet with
frame $(w_1,\ldots,w_k)$; this is by construction, starting with the
fact that $\unitJet$ is a unit jet.

Let $\lambda$ be the level of the draining reaction $y \to y'$.  By
Definition~\ref{d:pull}.3, $L = \lim\limits_{i\to\infty}
\theta(i)^{\beta_\lambda(i)}$ exists
in $[1,\infty]$.  The proof now breaks into two cases based on whether
$L = \infty$ or $1 \leq L < \infty$.

\paragraph*{Case 1 (Monomial domination):}
$\lim\limits_{i\to\infty} \theta(i)^{\beta_\lambda(i)} = \infty$.


The network $G$ is strongly endotactic, $w_1 \notin \HPerp$ by
hypothesis, and $(w_1, \ldots, w_\lambda)$ is a frame in $\bR^\SS$, so
by Proposition~\ref{p:stendjets} there exists a reaction $x \to x'$
in~$\RR$ that is $(w_1,\ldots,w_\lambda)$-sustaining with $x \in
\TT_\lambda$.  Hence the level of $x \to x'$ (which exists by
definition of sustaining) is at most~$\lambda$.  Also, by
Proposition~\ref{p:relrn}.1, $x \to x'$ is sustaining along the unit
jet $(v_\lambda(i))$.  Thus $x \to x'$ is sustaining along $(\wI)$
because $w(i) = \lVert \sum_{j=1}^k \beta_k(i) w_j \rVert v_k(i) +
O(\beta_{k+1}(i))$.  Thus it suffices to prove that the
limit~\eqref{eq:pullratio} holds for this reaction $x \to x'$.

Proposition~\ref{p:relrn}.2 and the definition of unit jet imply the
following asymptotics:
\begin{align*}
\frac{ |\<w(i),x'-x\>| }{\beta_\lambda(i) } = \Omega(1)
	\quad {\rm and } \quad
\frac{\beta_\lambda(i)}{|\<w(i),y'-y\>|} = \Theta(1).
\end{align*}
Thus the ratio of inner products satisfies
\begin{align}\label{eq:ipo1}
  \bigg|\frac{\<w(i),x'-x\>}{\<w(i), y'-y\>} \bigg| = \Omega(1).
\end{align}


Since $x\in\TT_\lambda$ and $y \in \source{\RR} \setminus \TT_\lambda$
(by Lemma~\ref{l:endjets}.2), they satisfy the hypotheses of
Lemma~\ref{l:mmlrat}.\ref{l:mmlrat.1}, so we obtain an index $k \leq
\lambda$ such that the following inequality holds for the monomial
term of interest (which also uses the fact that $\theta(i) > 1$ for
all~$i$):
\begin{align*}
\theta(i)^{\<w(i),x-y\>}
   &\geq \theta(i)^{ \beta_k(i) \<w_k, x-y\> + O(\beta_{k+1}(i))}
\notag
\\ &= \big(\theta(i)^{\beta_k(i)}\big)^{\<w_k, x-y\> +
      O(\beta_{k+1}(i)/\beta_k(i))}.
\end{align*}
This last quantity has limit $\infty$ as $i$ grows because
\begin{bullets}
\item%
$\<w_k, x-y\> > 0$ by definition of the index $k$ that arose from
Lemma~\ref{l:mmlrat}.\ref{l:mmlrat.1};
\item%
$\theta(i)^{\beta_k(i)} \geq \theta(i)^{\beta_\lambda(i)} \to \infty$
for large $i$ by assumption in this Case~1; and
\item%
$0 = O\big(\beta_{k+1}(i)/\beta_k(i)\big)$.
\end{bullets}
Therefore, combining with equation~\eqref{eq:ipo1}, the desired
limit~\eqref{eq:pullratio} holds.

\paragraph*{Case 2 (Inner product domination):}
$1 \leq \lim\limits_{i\to\infty} \theta(i)^{\beta_\lambda(i)}<\infty$.

The level of $y \to y'$ satisfies $\lambda > 1$ because otherwise
$\lim\limits_{i\to\infty} \theta(i)^{\beta_\lambda(i)} =
\lim\limits_{i\to\infty} \theta(i) = \infty$ by Lemma~\ref{l:w1} and
the definition of toric jet.
Consider, therefore, the unit jet $(v_{\lambda-1}(i))$ defined
in~\eqref{eq:seq_vk} with frame $(w_1, \ldots, w_{\lambda-1})$.  The
network $G$ is strongly endotactic, so as in Case 1, there exists a
reaction $x \to x'$ in $\RR$, with level $\leq \lambda - 1$ and $x \in
\TT_{\lambda-1}$, that is sustaining along $(v_{\lambda-1}(i))$ and
hence also along $(\wI)$.


Again by Proposition~\ref{p:relrn}.2,
$$%
\lim\limits_{i \to \infty}\frac{|\<w(i),x'-x\>|}{\beta_\lambda(i)} = \infty
\quad\text{and}\quad
\frac{\beta_\lambda(i)}{|\<w(i),y'-y\>|} = \Theta(1),
$$
the former because
$\displaystyle\frac{|\<w(i),x'-x\>|}{\beta_{\lambda-1}(i)} =
\Omega(1)$.
Therefore
\begin{align}\label{eq:ipo2}
\lim\limits_{i \to \infty}\bigg|\frac{\<w(i),x'-x\>}{\<w(i),y'-y\>}\bigg| = \infty.
\end{align}
As $x \in \TT_{\lambda-1}$ and $y \in \source\RR$,
Lemma~\ref{l:mmlrat}.\ref{l:mmlrat.2} implies the first inequality
here:
\begin{align*}
\theta(i)^{\<w(i),x-y\>}
   &\geq \theta(i)^{ \beta_\lambda(i) \<w_\lambda, x-y\> + O(\beta_{\lambda+1}(i))}
\\ &= \left(\theta(i)^{\beta_\lambda(i)}\right)^{\<w_\lambda, x-y\> +
      O(\beta_{\lambda+1}(i)/\beta_\lambda(i))}
\\ & \to \left( \lim\limits_{i \to \infty} \thetaI^{\beta_\lambda(i)}\right)
     ^{\<w_\lambda, x-y\> + 0}
     \qquad\text{because }\bCoef{\lambda}i\text{ dominates }\bCoef{\lambda+1}i
\\
   &\geq 1\hspace{30.5ex}
    \text{because }1\leq\lim\limits_{i\to\infty}\theta(i)^{\beta_\lambda(i)}<\infty.
\end{align*}
Therefore $\theta(i)^{\<w(i),x-y\>} = \Omega(1)$.  When combined
with~\eqref{eq:ipo2} this implies that the sustaining reaction $x \to
x'$ dominates $y \to y'$.  Hence the required
limit~\eqref{eq:pullratio} holds for Case 2.
\end{proof}


\begin{example}\label{ex:dom_lem_false_wo_Hperp}
Proposition~\ref{p:domination} is false without the assumption that
$w_1 \notin \HPerp$.  For example, consider the network consisting of
the single reversible reaction $A \lra B$.  The direction $u =
(1/\sqrt{2},1/\sqrt{2})$ is perpendicular to the stoichiometric
subspace, and $v = (-1/\sqrt{2},1/\sqrt{2})$ lies in the subspace.
Consider an adapted toric jet $\toricJet$ framed by $(u,v)$.  Then
the reaction $B \to A$ is sustaining along $(w(i))$, while the
reaction $A \to B$ is draining.  The ratio of their pulls is
\begin{align}\label{eq:ex_pullratio}
\bigg|\frac{\pull_{B \to A}(i)}{\pull_{A \to B}(i)}\bigg|
  = \theta(i)^{\<\beta_1(i)u + \beta_2(i)v, ~ (0,1)-(1,0)\>}
  = (\theta(i)^{\beta_2(i)})^{\<v, ~ (-1,1)\>}
  = (\theta(i)^{\beta_2(i)})^{\sqrt{2}} .
\end{align}
In particular, if $\beta_2(i)$ approaches~$0$ much faster than
$\theta(i)$ approaches~$\infty$, then the limit of
$\theta(i)^{\beta_2(i)}$ does not diverge, so the limit of the
ratio~\eqref{eq:ex_pullratio} is not~$\infty$, whence the sustaining
reaction does not dominate the draining reaction.
\end{example}


\begin{remark}\label{rmk:connection_Anderson}
The idea of using a \good reaction to dominate \bad reactions in the
proof of Proposition~\ref{p:domination} is similar to arguments in the
proof of Lemma~4.8 in the work of Anderson~\cite{Anderson11}.
Furthermore, Anderson's result is similar to our
Theorem~\ref{thm:LyFncWorks_cpc}, below: essentially, both results
show that a certain function decreases along trajectories outside a
compact set.
Anderson's concept of ``partitioning vectors $y \in \CC$ along a
sequence of trajectory points'' focuses first on the question of which
monomials $\theta^{\<w, y\>}$
dominate, and then later in the analysis analyzes the inner product 
$\<w, y'-y\>$, whereas we consider 
the entire product $\theta^{\<w, y\>}\<w, y'-y\>$ in the definition of 
pull~\eqref{eq:pull}.  Anderson's concept fits into ours in the following 
way.  For a sequence of
trajectory points, written as $\thetaI^{\wI}$, that do not remain in a
compact set, consider the ``top tier'' defined by the monomial or the
pull, respectively.  That is, the top tier is the set of reactions $y
\to y'$ such that the corresponding monomials or, respectively,
corresponding pulls dominate all others along the sequence.  For
strongly endotactic networks, these two top tiers coincide, with the
possible exception of reactions orthogonal to a limiting direction of
the vectors $(\wN)$.
\end{remark}

\begin{remark}\label{rmk:power}
Some ideas in the proof of Proposition~\ref{p:domination} are cognate
to ideas in Power Geometry~\cite{Bruno}.  Specifically, in Power
Geometry, to determine which terms in a polynomial dominate (for
instance, when certain coordinates go to zero of infinity) one works
in the log of the coordinates and examines which exponent vectors of
the polynomial lie in the relevant face of the Newton polytope of the
polynomial.
\end{remark}

\section{A Lyapunov-like function for strongly endotactic networks}\label{s:prelya}

This section uses the results on jets from the previous section to
show (Theorem~\ref{thm:LyFncWorks_cpc}) that for strongly endotactic
networks, outside a compact set the function $g(x) = \sum_{i\in \SS}
x_i \log x_i - x_i$ from Definition~\ref{d:Lyap} decreases along
trajectories.


\begin{definition}\label{d:preL}
If $N$ is a confined reaction system, specified by a reaction network
$\GG$, a tempering~$\kappa$, and an invariant polyhedron~$\invtPoly$,
then $g$ \emph{decreases along trajectories} of the mass-action
differential inclusion (Definition~\ref{d:MADiffI}) arising from~$N$
\emph{outside a compact set} if there exists a compact set $\Kset
\subseteq \relIntP$ such that for all trajectories $x(t)$, the time
derivative satisfies $\frac{d}{dt}g(x(t))|_{t = t^*} < 0$ whenever
$x(t^*) \in \relIntP \setminus \Kset$.
\end{definition}


\begin{remark}\label{rmk:preL}
Every strict Lyapunov function decreases outside the compact set
consisting of the function's unique minimum.
Functions that decrease outside a compact set should be compared with
Foster--Lyapunov functions used in the analysis of Markov
chains~\cite[Appendix B.1]{MT}.  Foster--Lyapunov functions are used,
for instance, to prove that a Markov chain always reaches a certain
set, the analogue of our compact set $\Kset$.
\end{remark}

The next lemma states that our compact set of interest, $\Kset =
\Kset_{\cutoff}$, is a compact subset of the positive orthant
$\Rplus^\SS$.  In contrast, some of the sublevel sets of the function
$g$ are not compact.  Indeed, it is the fact that some level sets of
$g$ intersect the boundary of the positive orthant that prevents us
from using the sublevel sets as our sets $\Kset$; see
Remark~\ref{rmk:whyNoPermanence}.


\begin{lemma}\label{l:cutoff_cpc}
Let $\invtPoly$ be an invariant polyhedron of a reaction network $\GG$,
and let $\theta> 1$.  Then the set
$$%
  \Kset_\theta
  =
  \big\{\theta_0^w \,\big|\, 1 \leq \theta_0 \leq \theta \text{ and } w \in
  \bR^\SS \text{ with } ||w|| = 1\big\} \cap \invtPoly
$$
is a compact subset of $\relIntP$.
\end{lemma}
\begin{proof}
$\Kset_\theta$ is compact because it is the intersection of a closed
set~$\invtPoly$ with a continuous image (under coordinatewise
exponential) of a compact set, namely, the closed ball of radius $\log
\theta$ around the origin in $\bR^\SS$.
The intersection is in $\relIntP$ because exponentials never vanish.
\end{proof}


\begin{lemma}\label{l:preL}
Let $N$ be a confined reaction system, specified by a network $\GG$, a
tempering $\kappa$, and an invariant polyhedron~$\invtPoly$.  Suppose
there exists $\cutoff>1$ such that
\begin{align}\label{eq:sum_pulls}
  \sum_{r = (y \to y') \in \RR} k_r \param^{\<w,y\>} \<w, y' - y\> < 0
\end{align}
for all $\theta_0 > \cutoff$ and unit vectors $w$ in $\bR^\SS$ with
$\theta_0^w \in \invtPoly \setminus \Kset_\theta$, and for all
$(k_r)_{r\in\RR}\in\prod_{r\in\RR}\kappa(r)$.  Then outside the set
$\Kset_\cutoff$ in Lemma~\ref{l:cutoff_cpc}, the function $g(x)=
\sum_{i\in\SS} x_i \log x_i - x_i$ decreases along trajectories of the
mass-action differential inclusion (Definition~\ref{d:MADiffI})
arising from~$N$.
\end{lemma}
\begin{proof}
We need $\frac{d}{dt} g (x(t)) |_{t=t^*} < 0$ for any trajectory point
$x(t^*) \in \relInt(\invtPoly) \setminus \Kset_\cutoff$.  Such a
trajectory point can be written as $x(t^*) = \theta_0^w$ where $\param
> \cutoff$ and $w$ is a unit vector in $\bR^\SS$, by
Lemma~\ref{l:cutoff_cpc}.  Additionally, $\frac{d}{dt}
x(t)|_{t=t^*}$ has the form given in Definition~\ref{d:MADiffI} for
some rates $k_r \in \kappa (r)$.  The following computation is
straightforward, using of the gradient of~$g$:
\begin{align*}
  \frac{d}{dt} g (x(t)) |_{t=t^*}
  & =
  \Big\<\nabla g(\param^w), \sum_{r\in\RR}k_r (\param^w)^{\source{r}}
  \flux{r}\Big\>
  \notag
\\
  & =
  \Big\<(\log \param)w, \sum_{r\in \RR} k_r \param^{\<w,\source{r}\>}
  \flux{r}\Big\>
  \notag
\\
  & =
  (\log \param) \sum_{r = (y \to y' ) \in \RR} k_r \param^{\<w,y\>}
  \<w, y' - y\>.
\end{align*}
As $\log \param > 0$ because $ \param > 1$, the desired inequality
$\frac{d}{dt} g (x(t)) |_{t=t^*} < 0$ follows
from~\eqref{eq:sum_pulls}.
\end{proof}


\begin{theorem}\label{thm:LyFncWorks_cpc}
If $N$ is a confined strongly endotactic reaction system, specified by
a reaction network $\GG$, a tempering $\kappa$, and an invariant
polyhedron~$\invtPoly$, then outside a compact set
$\Kset\subseteq\bR^\SS_{>0}$, the function $g(x)= \sum_{i \in \SS} x_i
\log x_i - x_i $ decreases along trajectories of the mass-action
differential inclusion (Definition~\ref{d:MADiffI}) arising from~$N$.
\end{theorem}
\begin{proof}
Begin by fixing a number $0 < \varepsilon < 1$ such that $\varepsilon$
is a lower bound for the tempering $\kappa$ and $1/\varepsilon$ is an
upper bound: for all reactions $r \in \RR$, $\kappa(r) \subseteq
(\varepsilon, 1/\varepsilon)$.
The goal is to demonstrate the existence of a cutoff $\cutoff > 1$
such that $g$ decreases along trajectories outside of the compact
subset $\Kset_\theta \subseteq \relIntP$ from
Lemma~\ref{l:cutoff_cpc}.  The proof proceeds by contradiction:
assuming that no such $\cutoff$ exists, construct an impossible toric
jet.

To be precise, assume $g(x)$ fails to decrease along trajectories
outside of each set $\Kset_\theta$.  Pick a sequence $\thetaN \to
\infty$ of real numbers $> 1$, so $\Kset_{\theta(1)} \subseteq
\Kset_{\theta(2)} \subseteq \cdots$ and $\bigcup_i \Kset_{\theta(i)} =
\Rplus^\SS$.
Lemma~\ref{l:preL} grants
\begin{bullets}
\item%
a sequence of points $x(i) = \thetaN^{\wN} \in \invtPoly$, where $\wN$
is a unit vector in $\bR^\SS$, and
\item%
for each reaction $r \in \RR$, a sequence $k_r(i)$ of rates in the
interval $\kappa(r)$
\end{bullets}
such that for all~$i$,
\begin{equation}\label{eq:sum_pulls_violation}
\sum_{r=(y\to y')\in\RR}k_r(i)\thetaN^{\<\wN,y\>}\<w(i),y'-y\> \geq 0.
\end{equation}
Lemma~\ref{l:jet} produces a subsequence of $(w(i))$ that is a unit
jet; call its jet frame $(w_1, \ldots, w_\ell)$ in~$\bR^\SS$ and, as
usual, denote the subsequence again by $(w(i))$, for ease of notation.
The resulting sequence $\toricJet$ is a toric jet by construction.
Lemma~\ref{l:jet} affords a subsequence that is adapted to $\GG$, and
it is this toric jet $\toricJet$ whose impossibility we demonstrate by
appeal to Proposition~\ref{p:domination}, which leverages the strongly
endotactic hypothesis on $\GG$ to deduce the opposite of the
inequality~(\ref{eq:sum_pulls_violation}).
	
Both $k_r(i)$ and $\thetaN^{\<\wN, y\>}$
in~\eqref{eq:sum_pulls_violation} remain strictly positive.  The only
negative contributions to the sum in~\eqref{eq:sum_pulls_violation}
are from reactions that are \good along $\unitJet$, whereas \bad
reactions contribute positively.  To
contradict~\eqref{eq:sum_pulls_violation}, it suffices to show that
for any reaction $y \to y'$ that is \bad along $\unitJet$, the
contribution to the sum~\eqref{eq:sum_pulls_violation} of some \good
reaction $x \to x'$ eventually dominates that of $y \to y'$ by at
least a factor of $|\RR|$, the number of reactions of the network.  In
other words, for a fixed \bad reaction $r = (y \to y')$, it suffices
to exhibit a \good reaction $r_{\mathrm{sus}} = (x \to x')$ such
that~the~inequality
\begin{align}\label{eq:good_vs_bad}
  \frac{k_{r_\mathrm{sus}}(i)}{k_r(i)} \thetaN^{\<\wN, x-y\>}
  \frac{\<w(i), x'-x\>}{\<w(i), y-y'\>}
  \ >\ |\RR|
\end{align}
holds for large~$i$.
The number~$\varepsilon$ was constructed early in the proof so that
$k_{r_\mathrm{sus}}(i) > \varepsilon$ and $k_r(i) < 1/\varepsilon$.
Consequently, the desired inequality~\eqref{eq:good_vs_bad} follows
from the inequality
$$%
\left\lvert \frac{\pull_{x \to x'}(i)}{\pull_{y \to y'}(i)} \right\rvert
  \ =\
  \thetaN ^ {\<\wN ,~x - y\>} ~ \frac{\<w(i), x' - x\>}{\<w(i), y-y'\>}
  \ >\
  \frac{|\RR|}{\varepsilon^2}.
$$
It is sufficient (but not necessary) to prove that the left-hand side
of this inequality
has limit~$\infty$, i.e., the draining reaction is dominated by some
sustaining reaction along this adapted toric jet.  This follows from
Proposition~\ref{p:domination}, completing the proof, once $w_1$ is
verified not to be orthogonal to the stoichiometric subspace~$H$.


To prove that $w_1 \notin \HPerp$, first express each vector~$w_j$
of the jet frame uniquely as a sum $w_{j,H} + w_{j,\HPerp}$ of a
vector $w_{j,H} \in H$ and a vector $w_{j,\HPerp} \in \HPerp$.
Birch's theorem (Theorem~\ref{t:birch}) produces a unique point $q$ in
the intersection $\op{int}(\invtPoly) \cap \{\theta^w \mid \theta \in
\Rplus \text{ and } w \in\nolinebreak H^\perp\}$.
Pick a neighborhood $\Oset$ in $\relIntP$ around $q$ whose closure
$\ol\Oset$ is also contained in~$\relIntP$.  The sequence $\toricJet$
eventually avoids $\Oset$, because $x(i)$ approaches the boundary of
the closure $\ol\invtPoly$ of~$\invtPoly$ in the compactification
$[0,\infty]^\SS$ (see~Remark~\ref{rmk:testSeq}).  Thus, by
Corollary~\ref{cor:bddStoicSubComponent}, there exists a number $0 <
\mu \leq 1$ such that the $H$-components $w_H(i)$ of the unit vectors
$w(i)$ have norm at least $\mu$, for large~$i$:
$$%
\lVert w_H(i) \rVert
  = \lVert\bCoef{1}{i} w_{1,H} + \cdots + \bCoef{\ell}{i} w_{\ell,H}\rVert
\geq\mu.
$$
Since the first coefficient $\bCoef{1}{i}$ approaches~$1$ while all
others approach~$0$, the left-hand side of this inequality
has limit $\lVert w_{1,H}\rVert$, so $\lVert w_{1,H} \rVert \geq \mu
>0$.  Thus $w_{1,H} \neq 0$, as desired.
\end{proof}

\begin{remark}
We wish to emphasize that our extensions of Birch's Theorem 
(in particular, Corollary~\ref{cor:bddStoicSubComponent}) were used 
in the proof of Theorem~\ref{thm:LyFncWorks_cpc} to show that $w_1 
\notin \HPerp$.  This novel use of Birch's Theorem is one of our main 
contributions.  
\end{remark}


\section{Main results: persistence and permanence}\label{s:results}

In this section, we prove that strongly endotactic networks have
bounded trajectories (Theorem~\ref{thm:bddness}), are persistent
(Theorem~\ref{thm:persistence}), and are permanent
(Theorem~\ref{thm:main}).  Although the permanence result is stronger
than the first two, the two weaker results are applied in the proof of
permanence.  In this section, we rely on two key prior results:
\begin{bullets}
\item%
the decrease of the pseudo-Helmholtz free energy function along
trajectories outside a compact set (Theorem~\ref{thm:LyFncWorks_cpc}),
and
\item%
a projection argument from our earlier work~\cite{ProjArg}.
\end{bullets}
Additionally,
we prove the existence of steady states for strongly endotactic
networks in the mass-action ODE setting.

Throughout this section, Notation~\ref{not:products} remains in
effect, as does Definition~\ref{d:Lyap}, which includes the special
case $g(x) = g_{(1,\ldots,1)}(x)= \sum_{i=1}^{|\SS|} x_i \log x_i -
x_i$ from~\cite{Feinberg72,Horn72,HornJackson}.


\subsection{Boundedness of trajectories}\label{sub:boundedness}

Conjecture~\ref{conj:boundedness} holds for strongly endotactic
networks.


\begin{theorem}\label{thm:bddness}
Strongly endotactic networks have bounded trajectories.  More
precisely, the image of every trajectory of every mass-action
differential inclusion (Definition~\ref{d:MADiffI}) arising from a
confined strongly endotactic reaction system is bounded.
\end{theorem}
\begin{proof}
For a mass-action differential inclusion of a confined strongly
endotactic reaction system, specified by $\GG$, a tempering $\kappa$,
and an invariant polyhedron~$\invtPoly$, let $x:I \to\nolinebreak
\relIntP$ denote a trajectory.  By Theorem~\ref{thm:LyFncWorks_cpc},
there exists a compact set $\Kset \subseteq \relIntP$ such that
$\frac{d}{dt} g(x(t))|_{t=t^*} < 0$ for all trajectory points $x(t^*)
\in \relIntP \setminus \Kset$.  Let $M_1 = \sup\{g(x) \mid x \in
\Kset\}$ be the supremum of $g$ on $\Kset$, which is finite by
compactness.  Let $M$ be the maximum of $M_1$ and $g(\inf I)$.  Then
$x(t)$ remains in the bounded sublevel set $\mathcal{Q} = \{x \in
\invtPoly \mid g(x) \leq M\}$.  Indeed, the trajectory begins
in~$\mathcal{Q}$ by construction; when it is in $\mathcal{Q} \setminus
\Kset$, the value of $g$ decreases; and while it is in $\Kset$, the
value of $g$ can not exceed $M_1$ and thus can not exceed~$M$.
\end{proof}


\subsection{Persistence}\label{sub:pers}

Our next main result is the persistence of strongly endotactic
networks (Theorem~\ref{thm:persistence}).  The proof appeals to the
following lemma, a special case of~\cite[Theorem~3.7]{Anderson08} of
Anderson or~\cite[Proposition 20]{TDS} of Craciun et al.  It again
concerns the function $g(x) = \sum_{i=1}^m x_i \log x_i - x_i $, which
extends continuously to the boundary of the nonnegative orthant.


\begin{lemma}\label{l:vtx}
The function $g(x) = \sum_{i=1}^m x_i \log x_i - x_i$ defined on
$\bR_{\geq 0}^m$ has a local maximum at the origin.
\end{lemma}

Another result crucial to our proof of Theorem~\ref{thm:persistence}
is the next lemma, which follows from our earlier work on
``vertexical'' families of differential inclusions on hypercubes that
are well-behaved under maps \cite{ProjArg}.


\begin{lemma}\label{l:projArg}
Let $\mathcal{F}$ be the class of all confined strongly endotactic
reaction systems.  If
\begin{enumerate}
\item%
every mass-action differential inclusion of every reaction system
in~$\mathcal{F}$ is repelled by the origin
(Definition~\ref{d:persistent_permanent}), and
\item%
every trajectory of such a differential inclusion is bounded,
\end{enumerate}
then every strongly endotactic reaction network is persistent.
\end{lemma}
\begin{proof}
This is a special case of Corollary~6.4 of~\cite{ProjArg}.
\end{proof}

To be useful, Lemma~\ref{l:projArg} requires the following auxiliary
result.


\begin{lemma}\label{l:avoid_origin}
If the function $g(x) = \sum_{i \in \SS} x_i \log x_i - x_i$ for a
confined reaction system $N$ decreases along trajectories outside a
compact set, then the mass-action differential inclusion
(Definition~\ref{d:MADiffI}) arising from~$N$ is repelled by the
origin.
\end{lemma}

\begin{proof}
Let $O_1$ be a relatively open neighborhood of the origin $\ol 0 \in
\Rnn^\SS$, and let $\invtPoly$ be an invariant polyhedron of~$\GG$.
Let $\Kset \subseteq \relIntP$ denote a compact set outside of which
$g$ decreases along trajectories.  If $\invtPoly$ does not
contain~$\ol 0$, then all trajectories
avoid a fixed neighborhood of the origin because they remain
in~$\invtPoly$.  Therefore assume $\invtPoly$ contains~$\ol 0$.
Let $N_\varepsilon$ denote the intersection of the nonnegative orthant
$\Rnn^\SS$ with an open ball of radius~$\varepsilon$ around the
origin.  Let $\varepsilon > 0$ be sufficiently small so that
\begin{bullets}
\item%
$O_1$ contains $N_\varepsilon$,
\item%
$\varepsilon$ is less than the distance between the origin $\ol 0$
and the compact set $\Kset$, and
\item%
the maximum of $g$ on $N_\varepsilon$ is attained uniquely at the
origin (cf.\ Lemma~\ref{l:vtx}).
\end{bullets}
As $g(\ol 0)=0$, it follows that $g$ is strictly negative on
$N_\varepsilon$.
Define $-M$ to be the maximum value of $g$ for nonnegative vectors of
norm $\varepsilon$, which is attained by compactness:
\begin{align}\label{eq:maxLyapNearOrigin}
  -M = \max \big\{g(z) \,\big|\, z \in \Rnn^\SS \text{ with } \lVert z
  \rVert =\varepsilon\big\}.
\end{align}
Thus, $-M < g(\ol 0) = 0$.
Again by Lemma~\ref{l:vtx}, there exists $\delta > 0$ so that $g(z)
> -M/2$ for $z \in N_\delta$.
It is enough to show that any trajectory $x: I \to \relIntP$ that
begins outside $O_1$, and thus outside $N_\varepsilon$ as well, never
enters $N_\delta$.  Indeed, if the trajectory ever enters the closure
of $N_\varepsilon$ at some point $x(t^*)$,
then $g(x(t^*)) \leq -M$ by~\eqref{eq:maxLyapNearOrigin}.  Thus, since
$g(x(t))$ decreases while $x(t)$ is in $N_\varepsilon$, because
$N_\varepsilon \cap \Kset$ is empty, $x(t)$ never reaches $N_\delta$
by definition of~$\delta$.
\end{proof}

We can now prove that Conjecture~\ref{conj:persistence} holds for
strongly endotactic networks networks.


\begin{theorem}\label{thm:persistence}
Strongly endotactic networks are persistent.
\end{theorem}
\begin{proof}
Follows from Lemmas~\ref{l:projArg} and~\ref{l:avoid_origin} and
Theorems~\ref{thm:LyFncWorks_cpc} and~\ref{thm:bddness}.
\end{proof}


In addition, the \GAC holds for strongly endotactic complex-balanced
systems.

\begin{proof}[Proof of Theorem~\ref{thm:main2}]
Use Theorem~\ref{thm:persistence} and the discussion before
Conjecture~\ref{conj:persistence}.
\end{proof}

\begin{remark}\label{rmk:if_reflections_functorial}
Our reliance
on the projection argument (Lemma~\ref{l:projArg}) requires showing
that for strongly endotactic networks, the origin is repelling and
trajectories are bounded.  Instead of Lemma~\ref{l:projArg}, we could
have appealed to a related result, Corollary~6.3 in~\cite{ProjArg},
which states that for a ``vertexical'' family of differential
inclusions---such as the family arising from strongly endotactic
networks---if all vertices of $[0, \infty]^\SS$ are repelling, then
the boundary of $[0, \infty]^\SS$ is repelling and hence these
networks are persistent.  However, it is not known whether non-origin
vertices of $[0, \infty]^\SS$ are repelling for strongly endotactic
networks.
\end{remark}


\subsection{Permanence}\label{sub:perm}

Lemma~\ref{l:projArg} only records half of Corollary~6.4
of~\cite{ProjArg}.  The other half deals with permanence.

\begin{lemma}\label{l:projArg'}
In Lemma~\ref{l:projArg}, assume that 1 and 2 hold, and that $X$ is
such a differential inclusion on $\Rplus^\SS$.  If $K \subseteq
\Rplus^\SS$ is a compact set for which there exists $A \in \Rplus$
such that every trajectory of~$X$ that starts in~$K$ remains bounded
above by $A$ in each coordinate for all time, then for some compact
set $K^+ \subseteq \Rplus^\SS$, no trajectory of~$X$ that begins
in~$K$ leaves~$K^+$.
\end{lemma}

Applying Lemma~\ref{l:projArg'} requires an auxiliary result that
makes use of the persistence result.


\begin{lemma}\label{l:meet_cpc_set}
Let $X$ be the mass-action differential inclusion of a confined
strongly endotactic reaction system $N$, specified by a reaction
network $\GG$ together with a tempering~$\kappa$ and an invariant
polyhedron~$\invtPoly$.  Then there exists a compact set $K$ in
$\relIntP$ such that the image of every trajectory of $X$ defined on a
ray meets $K$.
\end{lemma}
\begin{proof}
By Theorem~\ref{thm:LyFncWorks_cpc}, there exists a compact set $\Kset
\subseteq \relIntP$ outside of which $g$ decreases along trajectories.
Enlarge $\Kset$ to a closed $\varepsilon$-neighborhood of $\Kset$,
denoted by $\widetilde{\Kset}$, so that $\widetilde{\Kset} \subseteq
\relIntP$; thus $\widetilde{\Kset}$ is compact.
Let $x: I \to \relIntP$ be a trajectory of the mass-action
differential inclusion arising from $\GG$ with $\kappa$ and
$\invtPoly$, where $I = [a,\infty]$ is a ray.  The goal is to show
that the image of $x$ meets $\widetilde{\Kset}$.  If $x(a) \in
\widetilde\Kset$, then we are done.  Otherwise, there exists a
nonempty maximal subinterval $[a,b) \subset I$ on which $g(x(t))$ is
decreasing.  If $b < \infty$, then $x(b) \in \Kset$, so we are done.
The remaining case has $g(x(t))$ decreasing for all time $t > a$.
The closure of the trajectory $x(t)$ is a compact subset of
$\relIntP$, because this network has bounded trajectories
(Theorem~\ref{thm:bddness}) and is persistent
(Theorem~\ref{thm:persistence}).

Any trajectory $x(t)$ that fails to intersect $\widetilde\Kset$ is
bounded away from $\Kset$.  The time derivative $\frac{d}{dt} g(x(t))$
of such a trajectory is therefore negative and bounded away from~$0$
(otherwise the closure of the trajectory would meet~$\Kset$, forcing
the trajectory itself to meet $\widetilde\Kset$), causing $g(x(t))$ to
have limit $-\infty$ as $t \to \infty$.  But $g$ is bounded below by
$-|\SS|$, so we conclude that every trajectory $x(t)$ must visit the
compact set $\widetilde\Kset$.
\end{proof}

We now prove our main result, which states that strongly endotactic
networks are permanent, thereby resolving
Conjecture~\ref{conj:ext_permanence} in the strongly endotactic case.

\begin{proof}[Proof of Theorem~\ref{thm:main}]
First, strongly endotactic networks are persistent by
Theorem~\ref{thm:persistence}.  Next, by Lemma~\ref{l:meet_cpc_set},
for every mass-action differential inclusion $X$ arising from a
strongly endotactic network, there exists a compact set $K$ in the
positive orthant that meets every trajectory of $X$ defined on a ray.
For the purpose of studying eventual properties of trajectories such
as permanence, we therefore assume that each trajectory defined on a
ray begins in~$K$.  Lemma~\ref{l:projArg'} applies for strongly
endotactic networks by Lemma~\ref{l:avoid_origin} along with
Theorems~\ref{thm:LyFncWorks_cpc} and~\ref{thm:bddness}; it implies
that there exists a compact set $K^+$ in $\Rplus^\SS$ such that every
such trajectory eventually remains in $K^+$.  Thus $X$ is permanent.
\end{proof}

As a corollary, certain weakly reversible networks are permanent.


\begin{corollary}\label{cor:w-rev}
If each linkage class of a weakly reversible network $\GG$ has the same
stoichiometric subspace, namely that of $\GG$, then $\GG$ is permanent.
\end{corollary}
\begin{proof}
Follows from Corollary~\ref{cor:same_subspace} and
Theorem~\ref{thm:main}.
\end{proof}

Similarly, weakly reversible reaction networks with one linkage class
are permanent.


\begin{proof}[Proof of Theorem~\ref{thm:OneLClass}]
Follows from the Corollary~\ref{cor:one_l_class} and
Theorem~\ref{thm:main}.
\end{proof}

\begin{remark}\label{rmk:whyNoPermanence}
Is an appeal to our earlier work (Lemmas~\ref{l:projArg}
and~\ref{l:projArg'}) necessary to prove permanence?  Could the fact
that $g$ decreases along trajectories suffice to deduce
Theorem~\ref{thm:main}?  Such an approach seems not to work.  On one
hand, we can prove that maximal trajectories must approach a
prescribed compact set in $\invtPoly$, namely the smallest sublevel
set~$Q$ of $g$ that contains~$\Kset$.  On the other hand, $Q$ may meet
the boundary of~$\invtPoly$.  Thus the union of all $\omega$-limit
sets of all possible trajectories lies in $Q \cap \relIntP$, but its
closure in $\invtPoly$ could meet the boundary of $\relIntP$.  Another
approach to permanence is to determine whether $\Kset$ is an
attracting set.  However, this also seems not to work: although
persistence can be used to prove that every trajectory outside $\Kset$
must enter a neighborhood of $\Kset$ (denoted by $\widetilde{\Kset}$
in the proof of Lemma~\ref{l:meet_cpc_set}), the trajectory could exit
and re-enter $\widetilde{\Kset}$ infinitely~often.
\end{remark}


\subsection{Existence of positive steady states}\label{s:exist_equilibria}

In the setting of mass-action ODE systems, Deng et al.\ proved that
for weakly reversible networks, each invariant polyhedron contains a
positive steady state~\cite{Deng}.  As for two-dimensional endotactic
networks, Craciun, Nazarov, and Pantea~\cite[\S 6]{CNP} explained that
their permanence result, together with the Brouwer fixed-point
theorem~\cite[\S 55]{Munkres}, implies the existence of positive
steady states.  (Their result is complementary to that of Deng~et~al.,
because weakly reversible networks form a proper subset of endotactic
networks.)  This standard application of the Brouwer fixed-point
theorem to establish the existence of positive steady states in the
setting of reaction systems goes back at least to Wei in
1962~\cite{Wei}.

We too are able to prove the existence of positive steady states
(Corollary~\ref{cor:exist_equilibria}).  However, this is not
accomplished via the Brouwer fixed-point theorem, because we do not
readily obtain a forward-invariant set $U$ in $\Rplus^\SS$ that is
homeomorphic to a ball; in particular, the sets $K$ and $K^+$ in the
proofs of Lemma~\ref{l:meet_cpc_set} and Theorem~\ref{thm:main} do not
suffice.  Instead our proof relies on the following result concerning
fixed points of dynamical systems, a special case
of~\cite[Theorem~3]{RW} first proven in a general setting by
Srzednicki~\cite{Srz}.


\begin{lemma}\label{l:exist_steady_state}
Let $X$ be a continuous ODE system on a polyhedron $\invtPoly$ in
$\bR^n$ such that every trajectory of $X$ defined on a ray meets a
compact set $K$.  Then $X$ has a steady state.
\end{lemma}


\begin{corollary}\label{cor:exist_equilibria}
For a mass-action kinetics ODE system arising from a strongly
endotactic network, each invariant polyhedron contains a positive
steady state.
\end{corollary}
\begin{proof}
Let $X$ denote such a mass-action kinetics ODE system with invariant
polyhedron $\invtPoly$.  The proof of Lemma~\ref{l:meet_cpc_set}
shows that there exists a compact set $K$ in $\relIntP$ such that the
image of every trajectory of $X$ defined on a ray meets $K$.  Thus, by
Lemma~\ref{l:exist_steady_state}, $X$ has a steady state in $K$,
which is therefore in $\relIntP$.
\end{proof}


\section{Examples}\label{s:ex}

The examples here illustrate our results as well as their limitations.


\subsection{Strongly endotactic networks}\label{sub:strongly}

\begin{example}\label{ex:net_again}
By Theorem~\ref{thm:main}, the strongly endotactic network from
Example~\ref{ex:str_end}, comprised of the reactions $0 \to 3A+B$, $2A
\to B$, and $2B \to A+B$, is permanent.
\end{example}

Permanence of the network in Example~\ref{ex:net_again} already was
known by the results of Craciun, Nazarov, and Pantea~\cite{CNP} and
extensions by Pantea~\cite{Pantea}, which apply to endotactic networks
with stoichiometric subspace of dimension at most~$2$.  For the
following $3$-dimensional example, their results do not apply.


\begin{example}\label{ex:net3d}
The following network is strongly endotactic, so it is permanent:
$$%
  0 \to A \to B \to C \to 0.
$$
\end{example}

The weaker persistence result for
the network in Example~\ref{ex:net3d} follows from work of Angeli,
De~Leenheer, and Sontag~\cite{ADS11} because it has no siphons.  In
contrast, the following example contains critical siphons, so the
results of Angeli, De Leenheer, and Sontag do~not~apply.


\begin{example}\label{ex:wdne}
The following network is strongly endotactic, so it is permanent:
$$%
  2A \lra B \quad \quad 2B \lra C \quad \quad 2C \lra A.
$$
Combinatorially, the reactant polytope is a triangular prism; one
triangle is formed by $A$, $B$, and $C$, and the other is formed by
$2A$, $2B$, and $2C$.  Each of the three pairs of reversible reactions
lies along a diagonal of one of the quadrilateral boundary faces of
the polytope.
\end{example}

Persistence in
Example~\ref{ex:wdne} already was known by work of Johnston and Siegel
\cite{JohnstonSiegel}, which extends easily to the setting of
mass-action differential inclusions, because the unique critical
siphon $\{A,B,C\}$ can be shown to be weakly dynamically
non-emptiable.  In the next example, however, no previous results
apply to obtain persistence: it is a three-dimensional network for
which $\{A,B,C\}$ is a critical siphon but not weakly dynamically
non-emptiable.


\begin{example}\label{ex:only_us}
The following network is strongly endotactic and therefore permanent:
$$%
  4A \to A+B+C \to 4B \to 4C \to 2A+2C.
$$
The reactant polytope is a tetrahedron with vertices $A+B+C$, $4A$,
$4B$, and $4C$.  Each of the four reactions lies on or is a subset of
one of the six edges of the tetrahedron.
\end{example}


\subsection{Limitations of the geometric approach}\label{s:limit}

The geometric approach presented in this work has some limitations in
its ability to resolve
Conjectures~\ref{conj:GAC}--\ref{conj:boundedness} in full generality.
We explain these limitations via a simple example.


\begin{example}\label{ex:obs1}
The following network $G$ is endotactic but not strongly endotactic:
$$%
  0 \lra A \quad \quad  B \lra 2B.
$$
Define a toric jet adapted to $G$ in the following way: take any
sequence $(\thetaN)$ of numbers greater than 1 that limits to
infinity, and let $\wN = -(\frac 1{\thetaN},1)/\lVert(\frac
1\thetaN,1)\rVert$.  It is straightforward to verify that the
conclusion of Proposition~\ref{p:domination} fails for every
subsequence of the resulting toric jet $\toricJet$.  Namely, the pull
of the \bad reaction $0 \leftarrow A$ with respect to the toric jet
cannot be dominated by either of the two \good reactions $0 \to A$ or
$B \to 2B$.  Although results of Craciun, Nazarov, and
Pantea~\cite{CNP,Pantea} can prove that this two-dimensional system is
permanent, higher-dimensional generalizations of this network can not
be proved to be permanent or even persistent.

The limiting direction of the unit jet
here is $\lim\limits_{n \to \infty} w(n) = (0,-1)$, and this direction
$(0,-1)$ witnesses the fact that the network fails to be strongly
endotactic.  Namely, there is no reaction with reactant among the
$\leq_{(0,-1)}$-maximal reactants $\{0,A\}$ and product outside of
that set.  Our results for strongly endotactic networks rely on the
existence of such a \good reaction, so they are unable to prove
permanence or even persistence of this network or similar networks by
way of toric jets.
\end{example}

In our final example, the question of persistence is as yet unresolved.


\begin{example}\label{ex:net_open}
The following $3$-dimensional network is weakly reversible---and
therefore endotactic---but not strongly endotactic:
$$%
  A \to B \to C \to A \quad \quad 2A \lra 3B.
$$
The reactant polytope is a pyramid whose base has vertices $A$, $2A$,
$B$, and $3B$.  Three of the reactions form a triangle along edges of
the pyramid, while the remaining pair of reversible reactions lies on
the edge of the pyramid nonadjacent to the triangle.  The network has
a unique critical siphon, $\{A,B,C\}$, which is not weakly dynamically
non-emptiable.  Persistence and permanence
(Conjectures~\ref{conj:persistence} and~\ref{conj:ext_permanence})
remain open for this network.
\end{example}


\paragraph{Acknowledgements.}
MG was supported by a Ramanujan fellowship from the Department of
Science and Technology, India, and, during a semester-long stay at
Duke University, by the Duke MathBio RTG grant NSF DMS-0943760.  EM
was supported by NSF grant DMS-1001437.  AS was supported by an NSF
postdoctoral fellowship DMS-1004380.  The authors thank
David~F.~Anderson, Gheorghe Craciun, Matthew Johnston, and Casian
Pantea for helpful discussions, and Duke University, where many of the
conversations occurred.  The authors also thank two referees whose perceptive and
insightful comments improved this work.

\providecommand{\bysame}{\leavevmode\hbox to3em{\hrulefill}\thinspace}



\begin{thebibliography}{13}
\renewcommand{\itemsep}{-.5ex}
\small

\bibitem{Adleman2008}
Leonard Adleman, Manoj Gopalkrishnan, Ming-Deh Huang, Pablo Moisset,
  and Dustin Reishus, \emph{On the mathematics of the law of mass
  action}, preprint 2008.\\  {\tt http://arXiv.org:0810.1108}

\bibitem{AndersonBd11}
David~F. Anderson, \emph{Boundedness of trajectories for weakly
  reversible, single linkage class reaction systems},
  J. Math. Chem. \textbf{49} (2011), no.~10, 2275--2290.

\bibitem{Anderson08}
\bysame, \emph{Global asymptotic stability for a class of
  nonlinear chemical equations}, SIAM J. Appl. Math. \textbf{68}
  (2008), no.~5, 1464--1476.

\bibitem{Anderson11}
\bysame, \emph{A proof of the global attractor conjecture in the
  single linkage class case}, SIAM J. Appl. Math.
  \textbf{71} (2011), no.~4, 1487--1508.

\bibitem{AndersonShiu10}
David~F. Anderson and Anne~Shiu, \emph{The dynamics of weakly
  reversible population processes near facets}, {SIAM}
  J. Appl. Math. \textbf{70} (2010), 1840--1858.

\bibitem{Angeli08}
David Angeli, \emph{On modularity and persistence of chemical
  reaction networks}, in \emph{Decision and Control, 2008. CDC
  2008. 47th IEEE Conference on}, pages 2650--2655, Dec. 2008.


\bibitem{ADS09}
David Angeli, Patrick De~Leenheer, and Eduardo Sontag,
  \emph{Chemical networks with inflows and outflows: A positive linear
  differential inclusions approach}, Biotechnol. Progr. \textbf{25}
  (2009), 632--642.

\bibitem{ADS10}
\bysame, \emph{Graph-theoretic characterizations of monotonicity of
  chemical networks in reaction coordinates},
  J. Math. Biol. \textbf{61} (2010), 581--616.

\bibitem{ADS11}
\bysame, \emph{Persistence results for chemical reaction networks
  with time-dependent kinetics and no global conservation laws}, SIAM
  J. Appl. Math. \textbf{71} (2011), 128--146.

\bibitem{BanajiM}
Murad Banaji and Janusz Mierczynski, \emph{Global convergence in systems
  of differential equations arising from chemical reaction networks},
  preprint 2012.  \url{http://arXiv.org:1205.1716}

\bibitem{Birch}
 M.W. Birch, \emph{Maximum likelihood in three-way contingency tables},
 J. Roy. Stat. Soc. B Met. \textbf{25} (1963), 220--233.

\bibitem{Bruno}
A.D. Bruno.  \emph{Power Geometry in Algebraic and Differential Equations}.
Elsevier, Amsterdam, 2000.

\bibitem{TDS}
Gheorghe Craciun, Alicia Dickenstein, Anne Shiu, and Bernd Sturmfels,
  \emph{Toric dynamical systems}, J. Symb. Comput. \textbf{44} (2009),
  no.~11, 1551--1565.

\bibitem{Craciun2010some}
Gheorghe Craciun, Luis Garc\'ia-Puente, and Frank Sottile, \emph{Some
  geometrical aspects of control points for toric patches},
  Lect. Notes Comput. Sc., Revised selected papers of 7th
  International Conference on Mathematical Methods for Curves and
  Surfaces,
  \textbf{5862} (2010), 111--135.

\bibitem{CNP}
Gheorghe Craciun, Fedor Nazarov, and Casian Pantea, \emph{Persistence
  and permanence of mass-action and power-law dynamical systems},
  SIAM J. Appl. Math.
 \textbf{73} (2013), no.~1, 305--329.

\bibitem{Deng}
Jian Deng, Martin Feinberg, Chris Jones, and Adrian Nachman,
  \emph{On the steady states of weakly reversible chemical reaction
  networks}, preprint, 2011.  \url{http://arXiv.org:1111.2386}

\bibitem{DB}
Pete Donnell and Murad Banaji.
\emph{Local and global stability of equilibria for a class of
chemical reaction networks}, preprint, 2012.
\url{http://arXiv.org:1211.2153}

\bibitem{Fein87}
Martin Feinberg,
  \emph{Chemical reaction network structure and the stability of
  complex isothermal reactors I. The deficiency zero and deficiency
  one theorems}, Chem. Eng. Sci. \textbf{42} (1987), no.~10,
  2229--2268.

\bibitem{Feinberg72}
\bysame, \emph{Complex balancing in general kinetic systems},
  Arch. Rational Mech. Anal. \textbf{49} (1972), no.~3, 187--194.





\bibitem{Gatermann}
Karin Gatermann, \emph{Counting stable solutions of sparse polynomial
  systems}, Proceedings of an AMS-IMS-SIAM Joint Summer Research
  Conference on Symbolic Computation: Solving Equations in Algebra,
  Geometry, and Engineering, Mount Holyoke College, South Hadley, MA,
  June 11-15, 2000. Vol. 286. American Mathematical Soc., 2001.

\bibitem{Gnacadja08}
{Gilles Gnacadja}, \emph{Univalent positive polynomial maps and the
  equilibrium state of chemical networks of reversible binding
  reactions}, Adv. in Appl. Math. \textbf{43} (2009), no.~4,
  394--414.

\bibitem{GnacadjaPers}
\bysame, \emph{Reachability, persistence, and constructive chemical
  reaction networks (part III): a mathematical formalism for binary
  enzymatic networks and application to persistence},
  J. Math. Chem. \textbf{49} (2011), no.~10, 2158--2176.

\bibitem{cat}
Manoj Gopalkrishnan, \emph{Catalysis in reaction networks},
  B. Math. Biol. \textbf{73} (2011), no.~12, 2962--2982.

\bibitem{ProjArg}
Manoj Gopalkrishnan, Ezra Miller, Anne Shiu.
  \emph{A projection argument for differential inclusions, with
  application to mass-action kinetics}, SIGMA \textbf{9} (2013), 025, 25 pages.
	
\bibitem{Horn74}
F. Horn, \emph{The dynamics of open reaction systems}, in
  \emph{Mathematical aspects of chemical and biochemical problems and
  quantum chemistry, SIAM--AMS Proceedings, VIII}, pages 125--137,
  1974.

\bibitem{Horn72}
\bysame, \emph{Necessary and sufficient conditions for complex
  balancing in chemical kinetics}, {Arch. Ration. Mech. Anal.}
  \textbf{49} (1972), no.~3, 172--186.

\bibitem{HornJackson}
Friedrich J.~M. Horn and Roy Jackson, \emph{General mass action
  kinetics}, Arch. Ration. Mech. Anal. \textbf{49} (1972), 81--116.
 
\bibitem{JS11}
Matthew D. Johnston and David Siegel, \emph{Linear conjugacy of
  chemical reaction networks}, J. Math. Chem. \textbf{49} (2011), no.~7,
  1263--1282.

\bibitem{JohnstonSiegel}
\bysame, \emph{Weak dynamic non-emptiability
  and persistence of chemical kinetic systems}, SIAM
  J. Appl. Math. \textbf{71} (2011), no.~4, 1263--1279.

\bibitem{JSS}
Matthew D. Johnston, David Siegel, and G\'abor Szederk\'enyi,
  \emph{Dynamical equivalence and linear conjugacy of chemical
  reaction networks: new results and methods},
  MATCH--Commun. Math. Comput. Chem. \textbf{68} (2012), no.~2,
  443--468.

\bibitem{Lotka1920}
A. J. Lotka, \emph{Analytical Note on Certain Rhythmic Relations in
  Organic Systems}, Proc. Natl. Acad. Sci. U.S. \textbf{6} (1920),
  410--415.

\bibitem{MT}
Sean Meyn and Richard L. Tweedie, \emph{Markov chains and stochastic
  stability}, second edition. Cambridge University Press, 2009.

\bibitem{abelBinomial}
Ezra Miller, \emph{Theory and applications of lattice point methods
  for binomial ideals}, in Combinatorial Aspects of Commutative
  Algebra and Algebraic Geometry, Proceedings of the Abel Symposium
  (Voss, Norway, 1--4 June 2009), Abel Symposia, vol.~6, Springer
  Berlin Heidelberg, 2011, pp.~99--154.

\bibitem{MP}
Ezra Miller and Igor Pak, \emph{Metric combinatorics of convex polyhedra:
  cut loci and nonoverlapping unfoldings,} Discrete
  Comput. Geom. \textbf{39} (2008), no.~1--3, 339--388.

\bibitem{MR}
Stefan M\"uller and Georg Regensburger,
  \emph{Generalized mass action systems: Complex balancing equilibria
  and sign vectors of the stoichiometric and kinetic-order subspaces},
SIAM J. Appl. Math. \textbf{72} (2012), no.~6,
1926--1947.

\bibitem{Munkres}
James Munkres.  \emph{Topology}. Prentice Hall, 2000.

\bibitem{ASCB}
Lior Pachter and Bernd Sturmfels. \emph{Algebraic statistics for
  computational biology}. Cambridge University Press, 2005.

\bibitem{Norris}
James R. Norris. \emph{Markov chains}. Cambridge University Press,
  Cambridge series in statistical and probabilistic mathematics, 1998.

\bibitem{Pantea}
Casian Pantea, \emph{On the persistence and global stability of
  mass-action systems}, SIAM J. Math. Anal. \textbf{44} (2012), no.~3,
  1636--1673.

\bibitem{Petri66}
Carl Adam Petri, \emph{Communication with automata}, Ph.D.\
  dissertation, Technischen Hochschule Darmstadt (1962).

\bibitem{QDS}
Yuri Rabinovich, Alistair Sinclair, and Avi Wigderson,
  \emph{Quadratic dynamical systems}, Proceedings of the 33rd Annual
  Symposium on Foundations of Computer Science (1992), 304--313.

\bibitem{RW}
David Richeson and Jim Wiseman, \emph{A fixed point theorem for
  bounded dynamical systems}, Illinois J. Math. \textbf{46} (2002),
  no.~2, 491--495.

\bibitem{Savageau}
Michael A. Savageau and Robert Rosen, \emph{Biochemical systems
  analysis: a study of function and design in molecular biology.}
  Addison-Wesley, Reading, MA, 1976.

\bibitem{concordant}
Guy Shinar and Martin Feinberg, \emph{Concordant chemical reaction
  networks}, Math Biosci. \textbf{240} (2012), no.~2, 92--113.

\bibitem{ShiuSturmfels}
Anne Shiu and Bernd Sturmfels, \emph{Siphons in chemical reaction
  networks}, B. Math. Biol. \textbf{72} (2010), no.~6, 1448--1463.

\bibitem{JohnstonSiegel_stratum}
David Siegel and Matthew D. Johnston, \emph{A stratum approach to
  global stability of complex balanced systems},
  Dyn. Syst. \textbf{26} (2011), 125--146.

\bibitem{SM}
David Siegel and Debbie MacLean, \emph{Global stability of complex
  balanced mechanisms}, J. Math. Chem. \textbf{27} (2000), no.~1--2,
  89--110.

\bibitem{Srz}
Roman Srzednicki, \emph{On rest points of dynamical systems},
  Fund. Math. \textbf{126} (1985), no.~1, 69--81.


\bibitem{SH}
G\'abor Szederk\'enyi and Katalin M. Hangos, \emph{Finding complex
  balanced and detailed balanced realizations of chemical reaction
  networks}, J. Math. Chem. \textbf{49} (2011), no.~6, 1163--1179.


\bibitem{Wei}
J. Wei, \emph{Axiomatic treatment of chemical reaction systems},
  J. Chem. Phys. \textbf{36} (1962), 1578--1584.

\bibitem{Ziegler}
G\"unter Ziegler. \emph{Lectures on Polytopes.} Vol.~152 of Graduate
  Texts in Mathematics.  Springer--Verlag, New York, 1995.

\end{thebibliography}
\end{document}